\documentclass[11pt,reqno]{amsart}
\usepackage{amsmath, amssymb, amsfonts, amstext, amsthm, amscd}
\usepackage[mathscr]{eucal}
\usepackage{enumerate}
\usepackage{tikz,color,soul}
\usetikzlibrary{matrix,arrows}
\usepackage[utf8]{inputenc}
\usepackage[english]{babel}

\usepackage{mathrsfs}

\usepackage{verbatim}

\usepackage[new]{old-arrows}

\usepackage{longtable}
\usepackage{hyperref}
\usepackage[all]{xy}
\usepackage{ulem}
\usepackage{xcolor}

\usetikzlibrary{arrows.meta, positioning,arrows}

\newtheorem{thm}{Theorem}[section]
\newtheorem{prop}[thm]{Proposition}
\newtheorem{lemma}[thm]{Lemma}
\newtheorem{cor}[thm]{Corollary}
\newtheorem{defn}[thm]{Definition}
\newtheorem{ex}[thm]{Example}
\newtheorem{rmk}[thm]{Remark}
\newtheorem{question}[thm]{Question}

\numberwithin{equation}{section}

\newcommand{\C}{\mathbb C}

\newcommand{\R}{\mathbb R}
\newcommand{\Z}{\mathbb Z}

\def \mc{\mathcal}

\allowdisplaybreaks

\begin{document}

\baselineskip=15pt

\title[Seshadri constants of equivariant vector bundles]{Seshadri constants of equivariant vector bundles on toric varieties}

\author[J. Dasgupta]{Jyoti Dasgupta}

\address{Indian Institute of Science Education and Research, Pune, 
	Dr. Homi Bhabha Road,
	Pashan, Pune 411 008, India}

\email{jdasgupta.maths@gmail.com}

\author[B. Khan]{Bivas Khan}

\address{Indian Institute of Science Education and Research, Pune, 
	Dr. Homi Bhabha Road,
	Pashan, Pune 411 008, India}

\email{bivaskhan10@gmail.com}

\author[Aditya Subramaniam]{Aditya Subramaniam}

\address{Chennai Mathematical Institute, H1 SIPCOT IT Park, Siruseri, Kelambakkam 603103, 
	India}
\email{adityas@cmi.ac.in}



\date{}

\maketitle

\begin{abstract}
	 
	  We compute Seshadri constants of a torus equivariant nef vector bundle on a projective space satisfying certain conditions. As an application, we compute Seshadri constants of tangent bundles on projective spaces. We also consider equivariant nef vector bundles on Bott towers of height 2 (i.e. Hirzebruch surfaces) and Bott towers of height 3 respectively. Assuming some conditions on the minimal slope of the restrictions of these bundles to invariant curves, we give precise values of Seshadri constant at an arbitrary point. We also give several examples illustrating our results.
\end{abstract}

\section{Introduction}


Seshadri constants quantify local positivity of an ample line bundle on a projective variety. Seshadri constants for line bundles were introduced by Demailly \cite{Dem}. The foundational idea behind Seshadri constants is an ampleness criterion for line bundles given by Seshadri \cite[Theorem 7.1]{Har}.  Seshadri constants have turned out to be an interesting invariant. In general, Seshadri constants are hard to compute and a lot of research is aimed at finding estimates. Most of the existing work on Seshadri constants has been in the case
of surfaces (for example, see \cite{Sy,Ga,HM}). There are very few results known over higher dimensional varieties, such as  abelian varieties (for example, see \cite{Na, La, Ba,
	Deb}), toric varieties (for example, see \cite{DiRocco, HMP, It1,
	It2}), Fano varieties (for example, see \cite{BS,
	LZ}), and Grassmann bundles over curves (\cite{BHNN}). See \cite{B} for a survey of results for Seshadri constants for line bundles.

Seshadri constants for vector bundles were first defined by Hacon \cite{Hacon} (see Definition \ref{sch}). These constants also appeared in  the work of Beltrametti–Schneider–Sommese \cite{BSS, BSS1}. Recently, Fulger and Murayama \cite{FM} have introduced Seshadri constants in a relative setting generalizing both Demailly and Hacon's definitions (see Definition \ref{fulm}). 
Seshadri constants for ample vector bundles on smooth projective curves are computed in \cite{Hacon} and for nef vector bundles in \cite{FM}.  The Seshadri constant of a discriminant zero semistable vector bundle can be computed from the determinant line bundle (see \cite[Lemma 3.27]{FM}). Seshadri constants for direct sum of vector bundles are the minimum of Seshadri constants of the component bundles (see \cite[Lemma 3.31]{FM}). Recently, Seshadri constants of ample vector bundles on some surfaces have been studied in \cite{MR}.

The study of Seshadri constants for equivariant vector bundles was initiated in \cite{HMP}. An equivariant vector bundle on a toric variety is a vector bundle together with a lift of the action of the torus.  One advantage of working with an equivariant vector bundle is that the Mori cone of the projectivized bundle is rational polyhedral (see Section \ref{mcot}). Also, one can find restriction of equivariant vector bundles to invariant curves using Klyachko's classification theorem (see Section \ref{subsec5}). In \cite[Proposition 3.2]{HMP}, the authors have computed Seshadri constants for torus fixed points of a toric variety. Our aim is to compute Seshadri constant at arbitrary point.  

The purpose of this paper is twofold. First we compute Seshadri constants of a torus equivariant nef vector bundle on a projective space satisfying certain conditions. This in particular includes the class of equivariant nef uniform bundles. Secondly, we consider equivariant nef vector bundles on a particular class of nonsingular projective complex toric varieties, namely, \textit{Bott towers}. 
For an integer $n \, \ge \, 0$, a \textit{Bott tower of height $n$}
\begin{equation}\label{bn}
	X_n \,\longrightarrow\, X_{n-1}\,\longrightarrow\, \ldots \,\longrightarrow\, X_2\,
	\longrightarrow\, X_1 \,\longrightarrow\, X_0\,=\,\{\text{point} \}
\end{equation} 
is defined inductively as an iterated ${\mathbb P}^1$--bundle so that at the 
$k$-th stage of the tower,  $X_k$ is of the form
$\mathbb{P}(\mathcal{O}_{X_{k-1}}  \oplus \mathcal{L}_{k-1}$) 
for a line bundle $\mathcal{L}_{k-1}$ over $X_{k-1}$.
So $X_1$ is isomorphic to \(\mathbb{P}^1\), $X_2$  
is a Hirzebruch surface and so on. 
Bott towers were shown to be degenerations of Bott-Samelson varieties by Grossberg and Karshon (see \cite{GK}). We call any stage of the tower \eqref{bn} also as a Bott tower. The original goal was to compute Seshadri constants for equivariant nef vector bundles over Bott towers, generalizing the results for line bundles obtained in \cite{Sc_bott}. However, it is cumbersome to do certain computations in higher dimensions. We compute Seshadri constants of equivariant nef vector bundles under some hypothesis on their restriction to invariant curves, and we restrict our attention to Bott towers of height 2 and 3. We also provide several examples to illustrate our results.

In Section \ref{prelim}, we set the notation and include some preliminary results for reference. We briefly describe the Mori cone construction for projective bundles on toric varieties in Section \ref{mcot}. 
In Section \ref{scpn}, we compute Seshadri constants of nef equivariant bundles on projective spaces satisfying certain conditions. In Section \ref{hirz} and Section \ref{X3}, we consider equivariant nef vector bundles (with certain assumptions) on Bott towers of height 2 and 3 respectively, and compute the Seshadri constants.

\section*{Acknowledgement} 
We would like to thank Krishna Hanumanthu for suggesting this problem, many helpful discussions and his comments on the manuscript. The second author is supported by NBHM postdoctoral fellowship from DAE. The third author is partially supported by a grant from Infosys Foundation.
\section{Preliminaries}\label{prelim}

\subsection{Seshadri constants for vector bundles}
Let $X$ be a nonsingular complex projective variety  and  $\mathcal{E}$ be a nef vector bundle on $X$.  Let $\mathbb{P}(\mathcal{E} )$ be the projective bundle associated to $\mathcal{E}$ and let \begin{equation*}
	\pi:\mathbb{P}(\mathcal{E} )\rightarrow X
\end{equation*}
be the natural projection. Fix a closed point $x \in X.$ Let $\psi:Bl_x(X) \to X$ be the blow up at $x \in X$ with exceptional divisor $E$. We then consider the commutative square Figure \ref{figblowup}
\begin{figure}[h] 
	\[ \xymatrix{
		Y' \ar[r]^{\psi'}\ar[d]_{\pi'} & \mathbb{P}(\mathcal{E} )\ar[d]^{\pi}\\
		Bl_x(X)  \ar[r]^{\psi}       & X 
	} \]
\caption{}
\label{figblowup}
\end{figure}
where $Y_{x}:=\pi^{-1}(x)$ and $Y':=Bl_{Y_x}(\mathbb{P}(\mathcal{E} ))=\mathbb{P}(\psi ^*\mathcal{E})$.

Let $\xi:=\mathcal{O}_{\mathbb{P}(\mathcal{E} )}(1)$ be the tautological line bundle on $\mathbb{P}(\mathcal{E} )$. Note that the vector bundle $\mathcal{E}$ on $X$ is said to be nef (respectively, ample) if  $\xi$ is nef (respectively, ample) on  $\mathbb{P}(\mathcal{E} )$.
\begin{defn}\label{sch}
The Seshadri constant of $\mathcal{E}$ at $x\in X$ is defined as  
\begin{align*}
	\varepsilon(X, \mathcal{E},x):=\text{sup}\,\Bigl\{\lambda \in \mathbb{R}_{\geq 0} \mid \psi'^*\xi-\lambda \, \pi'^*E \rm\hspace{1mm} is \hspace{1mm} nef \hspace{1mm} on \hspace{1mm} Y'\Bigr\}.	
\end{align*}
\end{defn}
This definition was introduced by Hacon in \cite{Hacon} for ample vector bundles.

There is another formulation of $\varepsilon(X, \mathcal{E},x)$ due to Fulger and Murayama \cite{FM}, which we use in this article.

Let $\mathcal{C}_{\mathcal{E}, x}$ denote the set of irreducible curves on $\mathbb{P}(\mathcal{E} )$ that meet $Y_{x}=\pi^{-1}(x)$, but are not contained in the support of $Y_{x}$. We note that the curves in $\mathcal{C}_{\mathcal{E}, x}$ are precisely the irreducible curves $C$ on $\mathbb{P}(\mathcal{E} )$ such that ${\rm mult}_{x}\pi_*C>0$. Here  ${\rm mult}_{x}\pi_*C$ denotes the multiplicity of $\pi_*C$ at $x$.
\begin{defn}\label{fulm}
The Seshadri constant of $\mathcal{E}$ at $x$ is defined to be 
\begin{equation*}
	\varepsilon(X, \mathcal{E},x):=\, \inf\limits_{\substack{C \in \mathcal{C}_{\mathcal{E}, x}}} \frac{\xi\cdot
		C}{{\rm mult}_{x}\pi_*C} .
\end{equation*}
\end{defn}

In \cite[Remark 3.10]{FM}, it was shown that Definition \ref{sch} and Definition \ref{fulm} are equivalent formulations for Seshadri constants of nef vector bundles on complex projective varieties. For simplicity, we sometimes denote the Seshadri constant $\varepsilon(X, \mathcal{E},x)$ simply by \(\varepsilon( \mathcal{E},x)\) when the variety \(X\) is clear from the context. From Definition \ref{fulm}, one can deduce that 
\begin{equation}\label{FMcorp3.21}
	\varepsilon(\mathcal{E},x) =\, \inf\limits_{\substack{x \in C \subset X}}  	\varepsilon(\mathcal{E}|_C,x),
\end{equation}
(see proof of \cite[Corollary 3.21]{FM}).


The following result is a key ingredient in our computations.

\noindent
Fix an ample divisor \(H\) on a nonsingular projective variety \(X\) of dimension \(n\). For a torsion free sheaf $\mathcal{E}$ over \(X\), we have \(\text{deg } \mathcal{E}=c_{1}(\mathcal{E}) \cdot H^{n-1}\) and slope $\mu(\mathcal{E})=\frac{\text{deg } \mathcal{E}}{\text{rank}(\mathcal{E})}$.

\begin{prop}$($Vector bundles on curves$)$ \cite[Corollary 3.21]{FM}\label{FMcor3.21}
If $\mathcal{E}$ is a nef vector bundle on a nonsingular complex projective curve $X$, then 
$$\varepsilon(\mathcal{E},x)=\mu_{min}(\mathcal{E})$$ for all $x \in X$, where $\mu_{min}(\mathcal{E})$ denotes the smallest slope of any quotient of \(\mathcal{E}\). 
\end{prop}

\subsection{Restricting equivariant vector bundles to invariant curves}\label{subsec5}

We recall the recipe of restricting an equivariant vector bundle to an invariant curve from \cite[Section 5]{HMP}, \cite[Section 5]{DJS}. Let \(X=X(\Delta)\) be a nonsingular complete complex toric variety under the action of the torus \(T\), determined by the fan $\Delta$ in the lattice \(N \cong \Z^n \). The dual lattice is \(M:=\text{Hom}_{\Z}(N, \Z)\). Let $\Delta(k)$ denote the set of $k$-dimensional cones  of $\Delta$, and let $\sigma(k)$ denote the set of $k$-dimensional faces (subcones) of a cone $\sigma$ in $\Delta$. Let $\mathcal{E}$ be an equivariant vector bundle of rank \(r\) on \(X\). By a celebrated classification theorem of Klyachko \cite[Theorem 2.2.1]{Kly}, $\mathcal{E}$ corresponds to decreasing filtrations \((E, \{E^{\rho}(i)\}_{\rho \in \Delta(1), \, i \in \Z})\) on a vector space \(E\), which satisfy the following compatibility condition:
\begin{equation}\label{KCC}
	\begin{split}
		& \text{for each } \sigma \in \Delta(n), \text{ there exists a decomposition } \\
		&E= \bigoplus_{u \in \mathbf{u}(\sigma)} L^{\sigma}_u \text{ such that }  E^{\rho}(i)=\sum\limits_{\langle u, v_{\rho}  \rangle \geq i} L^{\sigma}_u \text{ for all } \rho \in \sigma(1).
	\end{split}
\end{equation}
Here \(v_{\rho}\) denotes the primitive generator of a ray $\rho \in \Delta(1)$. Since $\sigma$ is a maximal cone, the multiset $\mathbf{u}(\sigma) \subset M$ is uniquely determined by the vector bundle $\mathcal{E}$ and the cone $\sigma$ (see \cite[Corollary 2.3]{payne}). For each $\sigma \in \Delta(n)$, the multiset $\mathbf{u}(\sigma)$ will be called the associated characters of the equivariant vector bundle $\mathcal{E}$. These multisets play a crucial role in determining the restriction of the equivariant vector bundle $\mathcal{E}$ to invariant curves.

Consider an invariant curve \(C\) in \(X\) corresponding to the cone $\tau \in \Delta(n-1)$. Since \(X\) is complete, there are two maximal cones $\sigma$ and $\sigma'$ in $\Delta(n)$ that contain $\tau$. Also, note that \(C \cong \mathbb{P}^1\). Let the multisets corresponding to these maximal cones be given by 
\begin{equation*}
	\mathbf{u}(\sigma)=\{\mathbf{u}_1, \ldots, \mathbf{u}_r\} \text{ and } \mathbf{u}(\sigma')=\{\mathbf{u}'_1, \ldots, \mathbf{u}'_r\} \text{ respectively}.
\end{equation*} 

Then by \cite[Corollary 5.5, Corollary 5.10]{HMP}, the restriction $\mathcal{E}|_C$ splits equivariantly into a direct sum of line bundles
\begin{equation*}
	\mathcal{L}_{\mathbf{u}_1, \mathbf{u}'_1}|_C \oplus \ldots \oplus \mathcal{L}_{\mathbf{u}_r, \mathbf{u}'_r}|_C.
\end{equation*}
Here \(\mathcal{L}_{\mathbf{u}_i, \mathbf{u}'_i} \) denotes a line bundle over \(U_{\sigma} \cup U_{\sigma'}\) for \(i=1, \ldots, r\). Moreover, each individual summand is given by \[\mathcal{L}_{\mathbf{u}_i, \mathbf{u}'_i}|_C \cong \mathcal{O}_{\mathbb{P}^1}(a),\] where \(\mathbf{u}_i-\mathbf{u}'_i\) is \(a\) times the primitive generator of $\tau^{\perp}$ that is positive on $\sigma$.

\subsection{Bott towers}\label{bott}

In this section we briefly recall the construction of Bott towers and a few results about them. For details see \cite{BP}.

\subsubsection{Fan structure of a Bott tower}

Let $T \,\cong\,
\left(\C^*\right)^n$ be an algebraic torus. 
Define its character lattice $M\,:=\,\text{Hom}(T, \,\C^*) \,\cong\, \Z^n$ and
the dual lattice $N\,:=\,\text{Hom}_{\Z}(M, \Z)$. Let $\Delta_n$ be a
fan in $N_{\R}\,:=\,N \otimes_{\Z} \R$ which defines the toric variety
$X_n$ under the action of the torus 
$T$. The set of edges of $\Delta_n$ will be denoted by
$\Delta_n(1)$. Let $e_1,\,\ldots,\, e_n$ be the standard basis for $\mathbb{R}^n$. 
Consider the following vectors:
\begin{equation}\label{ci}
	\begin{split}
		& v_1\,=\,e_1,\, \ldots,\, v_n\,=\,e_n, \\
		& v_{n+1}\,=\,-e_1+c_{1, 2} e_2 + \ldots + c_{1, n} e_n,\\
		& \hspace{1 cm} \vdots\\
		& v_{n+i}\,=\,-e_i+c_{i, i+1} e_{i+1} + \ldots + c_{i, n} e_n, 1 \leq i <n,\\
		&v_{2n}\,=\,-e_n.
	\end{split}
\end{equation}
The fan $\Delta_n$ of $X_n$ is complete, and it consists of these $2n$
edges and $2^n$ maximal cones of dimension $n$ generated by these 
edges such that no cone contains both the edges $v_i$ and $v_{n+i}$
for $i\,=\,1,\, \cdots,\, n$.  
It follows that any \(k\)-th stage Bott tower arises from a collection
of integers \(\{c_{i,j} \}_{1 \leq i< j \leq n}\) as in \eqref{ci}. These integers are
called the \textit{Bott numbers} of the given Bott tower. 
In this paper we will restrict our attention to the case when the
Bott numbers $\{c_{i, j}\}_{\{1\leq i < j \leq n \}}$ are all positive
integers.


Let $D_i$ denote the invariant prime divisor corresponding to the 
edge $v_{n+i}$, and let $D'_i$ denote the invariant prime divisor corresponding to the edge $v_i$
for \(i\,=\,1,\, \ldots,\, n\). We have the following relations:
\begin{equation}\label{linequiv}
	D'_1 \sim_{\text{lin}} D_1,\,
	D'_i \sim_{\text{lin}} D_{i}-c_{1, i}D_{1}-\ldots-c_{i-1, i}D_{i-1}\ \text{ for } i\,=\,2,\, \cdots,\, n.
\end{equation}
The Picard group of the Bott tower is $\text{Pic}(X_n)\,=\,\Z
D_{1} \oplus \,\ldots\, \oplus \Z D_{n}$.

\subsubsection{Quotient construction of a Bott tower} \label{qc}

Let us recall the quotient construction of Bott tower (see \cite[Theorem 7.8.7]{BP} for details). The Bott tower \(X_n\) can be obtained as the 
quotient of
\[\{(z_1,w_1, \cdots, z_n, w_n) \,\in\, \mathbb{C}^{2n} \,\mid\
|z_i|^2+|w_i|^2 \neq 0,\ 1 \,\leq\, i \,\leq\, n\} \,\cong\, (\mathbb{C}^2 \setminus {0})^n\]
by the action of \((\mathbb{C}^*)^n\), where the action is given by 
\begin{equation*}
	\begin{split}
		(t_1,\, \cdots,\, t_n) \, \cdot \, &(z_1, \, w_1, \, \cdots, \, z_n, \, w_n) \,\\
		& =\,
	(t_1 z_1,\, t_1w_1,\, t_1^{-c_{1,2}} t_2z_2,\, t_2w_2, \,\cdots,\, t_1^{-c_{1,n}} t_2^{-c_{2,n}}
	\cdots t_{n-1}^{c_{n-1, n}} t_nz_n,\,t_nw_n ).	
	\end{split}
\end{equation*}
A point of \(X_n\) is denoted by the equivalence class \([z_1:w_1: \ldots: z_n: w_n]\).
Note that \(D'_i\) (respectively, \(D_i\)) is just the vanishing locus of the coordinate \(z_i\) (respectively,
\(w_i\)), i.e., \(D'_i\,=\,\mathbb{V}(z_i)\) (respectively, \(D_i\,=\,\mathbb{V}(w_i)\)) for \(1 \,\leq\, i \,\leq\,
n\) (see \cite[Example 5.2.5]{Cox}). 

For each \(1 \,\leq\, i \,\leq\, n\), there is a map 
\begin{equation*}
	X_i \,\longrightarrow\, X_{i-1},\text{ by sending } [z_1:w_1: \ldots: z_i:w_i] \,\longmapsto\,
	[z_1:w_1: \ldots: z_{i-1}:w_{i-1}]
\end{equation*}
together with a section given by
\begin{equation*}\label{qt_section}
	X_{i-1}\,\longrightarrow\, X_i,\ \ [z_1:w_1: \ldots: z_{i-1}:w_{i-1}] \,\longmapsto\,
	[z_1:w_1: \ldots: z_{i-1}:w_{i-1}:0:1]\, .
\end{equation*}

\subsubsection{Mori cone of a Bott tower }
We recall the description of the Mori cone of a Bott tower from \cite[Subsection 2.4]{Sc_bott}. Fix a point $x\,\in \,X_n$. Let $X_i^{(1)} \,:=\, X_i$ for every $1 \,\le\, i \,\le\, n$. We construct a sequence of Bott towers of decreasing height. Let $\pi_i\,:\, X_i \,\longrightarrow\, X_1$ be the composition of maps in \eqref{bn} for all $2 \,\le\, i \,\le\, n$. Define $X_i^{(2)} \,:=\,
\pi_i^{-1}(\pi_n(x))$. Then $$X_n^{(2)}\,\longrightarrow\, X_{n-1}^{(2)}\,\longrightarrow\, \cdots
\,\longrightarrow\,  X_2^{(2)} \, =\,\mathbb{P}^1$$ is a Bott tower of height \(n-1\) with positive Bott numbers. Note that $x \,\in\, X_n^{(2)}$. Let $\pi_{2,i}\,:\, X_i^{(2)} \,\longrightarrow\, X_2^{(2)}$ be the composition of
these map for every $3 \,\le\, i \,\le\, n$. Similarly, define $X_i^{(3)} \,:=\,\pi_{2,i}^{-1}(\pi_{2,n}(x))$. Proceeding this way, we obtain a special class of subvarieties $X_i^{(j)}$ of $X_n$ for every $1\,\le\, j \,\le \,i \,\le\, n$ which \textit{depend on the given point
	$x$}. Note that $x \in X_n^{(j)}$ for all $1 \, \le \, j \, \le \, n$ and $X_i^{(i)} \,= \,\mathbb{P}^1$ for each $1 \,\le\, i \,\le\, n$.

Let $$X_n^{(i)} \longrightarrow X^{(i)}_{n-1}\longrightarrow \ldots
\longrightarrow X^{(i)}_{i+1}\longrightarrow X^{(i)}_i=\mathbb{P}^1$$ be a Bott tower of height \(n-i+1\) for $1 \,\le\, i \,\le\, n$. Then there is a section map 
$X_{j-1}^{(i)} \,\longrightarrow\, X_{j}^{(i)}$ for every $i+1 \,\le\, j \,\le\, n$. 
For each $1\,\le\, i \,\le\, n$, let $\sigma_i\,:\, X_i^{(i)} \,\longrightarrow\, X_n^{(i)}$ be the composition
of these section maps. Define
\begin{equation}\label{gni}
	\Gamma_n^{(i)} \,:=\,\sigma_i( X_i^{(i)})\, \subset X_n^{(i)} .
\end{equation}
Note that $\Gamma_n^{(n)} \,=\, X_n^{(n)}$ and we simply denote $\Gamma_n^{(1)}$ by $\Gamma_n$. 

A Cartier divisor $D\,=\,\sum_{i=1}^{k}a_iD_{i}$ on $X_n$ is ample (respectively, nef) if and only if $a_i \,> 0 \text{ (respectively, $a_i \ \geq \ 0$) }$ for all 
$i\,=\,1,\, \cdots, \,n$ (see \cite[Theorem 3.1.1, Corollary 3.1.2]{KD}).

\begin{prop}\cite[Proposition 2.3]{Sc_bott} \label{dual-basis}
	The curves $\Gamma_n,\, \Gamma_n^{(2)},\,\cdots,\,\Gamma_n^{(n)}$ defined in \eqref{gni} span 
	$\overline{\rm{NE}}(X_n)$, and they are dual to $D_1,\, \cdots ,\, D_n$. 
\end{prop}

We conclude this subsection with the following properties of these curves defined in \eqref{gni}, which will be needed in later sections while computing Seshadri constants:
\begin{equation}\label{curveGam}
\begin{split}
	&  x \in \Gamma_n^{(i)} \text{ implies that }  x \text{ is also in } \Gamma_n^{(i+1)} \text{ for } i=1, \ldots, n ;\\
	&  \text{if } x \,\in\, \Gamma_n^{(i)}\setminus \Gamma_n^{(i-1)} \, \text{and}\, C \, \text{is a curve containing}\, x, \text{ then } C \,\not\subset\, D_i', \, \text{for any } i =2 \ldots, n
\end{split}
\end{equation}
(see \cite[Lemma 2.4]{Sc_bott}).

\section{Mori cone of projective bundles on toric varieties}\label{mcot} 

One of the key ingredient in the computation of Seshadri constant of a vector bundle is the description of the Mori cone of the projectivized bundle. In case of equivariant vector bundles, it was remarked in \cite[Remark 2.5]{HMP} that the Mori cone is rational polyhedral. In this section, we briefly recall the description of Mori cone of projectivization of equivariant vector bundles on toric varieties.

Let $\mathcal{E}$ be an equivariant nef vector bundle of rank $r$ on a nonsingular complete complex toric variety \(X\). Let $\mathbb{P}(\mathcal{E} )$ be the projective bundle associated to $\mathcal{E}$ and \begin{equation*}
	\pi:\mathbb{P}(\mathcal{E} )\rightarrow X
\end{equation*}
be the natural projection. Let $\xi:=\mathcal{O}_{\mathbb{P}(\mathcal{E} )}(1)$ be the tautological line bundle on $\mathbb{P}(\mathcal{E} )$.
Let $l_1, \cdots, l_{\mathfrak{n}}$ be the invariant curves on $X$. 
The following result was proved in \cite[Theorem 2.1]{HMP}:
\begin{prop}\label{HMPT2.1}
	The equivariant vector bundle $\mathcal{E}$ is nef if and only if $\mathcal{E}|_{l_i}$ is nef for all \(i=1, \ldots, \mathfrak{n}\).
\end{prop}

Since each $l_i$ is isomorphic to $\mathbb{P}^1$, we have that $\mathcal{E}|_{l_i}$ is a direct sum of line bundles. This means that $\mathbb{P}(\mathcal{E}|_{l_i})$ has the structure of a toric variety. Let  \begin{equation*}
	\pi_i:\mathbb{P}(\mathcal{E}|_{l_i} )\rightarrow l_i
\end{equation*}
be the natural projection. Let $\xi_i:=\mathcal{O}_{\mathbb{P}(\mathcal{E}|_{l_i} )}(1)$ and $f_i$ denote the numerical equivalence class of a fibre of the map $\pi_i$ (see the commutative diagram Figure \ref{scfig}).

\begin{figure}[h] 
	\[ \xymatrix{
		\mathbb{P}(\mathcal{E}|_{l_i}) \ar@{^{(}->}[r]^{\eta_i}\ar[d]_{\pi_{i}} & \mathbb{P}(\mathcal{E} )\ar[d]^{\pi}\\
		l_i \ar@{^{(}->}[r]       & X
	} \]
	\caption{}
	\label{scfig}
\end{figure}

Using \cite[Lemma 2.1]{Fulger}, we have the following description of the Nef cone of $\mathbb{P}(\mathcal{E}|_{l_i})$: 
\begin{equation}\label{restn}	
	\text{Nef}(\mathbb{P}(\mathcal{E}|_{l_i}))=\Big\{a(\xi_i-\mu_{min}(\mathcal{E}|_{l_i})f_i)+bf_i\mid a,b \in \mathbb{R}_{\geq0}\Big\}.
\end{equation}  
From the Chern relation of \(\xi_{i}\), we have \(\xi_{i}^r= \text{deg }(\mathcal{E}|_{l_i})\).
Set  \begin{equation*}
	\begin{split}
		&m_i:=\text{deg}(\mathcal{E}|_{l_i})-\mu_{\text{min}}(\mathcal{E}|_{l_i}) \text{ for }i=1, \, \ldots, \, \mathfrak{n}.
	\end{split}
\end{equation*}
We have the following intersection product on \(\mathbb{P}(\mathcal{E}|_{l_i} )\):
\begin{equation}\label{prodcurve}
	\begin{split}
		&\xi_{i} \cdot (\xi_{i}^{r-2}f_i)=1, \\
		&\xi_{i} \cdot (\xi_{i}^{r-1}-m_i \, \xi_{i}^{r-2}f_i)=\mu_{\text{min}}(\mathcal{E}|_{l_i})
	\end{split}
\end{equation} (see \cite[Equation (1.1)]{Fulger}).
Dualizing \eqref{restn}, we get the description of the Mori cone 
\begin{equation}\label{restnmori}
	\overline{\rm{NE}}(\mathbb{P} (\mathcal{E}|_{l_i}))=\Big\{a(\xi_{i}^{r-2}f_i)+b(\xi_{i}^{r-1}-m_i \, \xi_{i}^{r-2}f_i) \mid a,b \in \mathbb{R}_{\geq 0} \Big\}.
\end{equation}
We note that for each $i$, 
\begin{equation}\label{pushforward}
	\begin{split}
		&\xi_{i}^{r-2}f_i= \Sigma_i, \, \xi_{i}^{r-1}-m_i \, \xi_{i}^{r-2}f_i=\Omega_i,
	\end{split}
\end{equation}
where \(\Sigma_i\) is an invariant fiber curve and \(\Omega_i\) is an invariant section curve in \(\mathbb{P} (\mathcal{E}|_{l_i})\) (see \cite[Remark 3]{DiRocco_defect}).
We fix the following notation:
\begin{equation}\label{mor}
	\begin{split}
		&C_0:= \eta_i(\Sigma_i), \, C_i:=\eta_i(\Omega_i)\text{ for }i=1, \, \ldots, \, \mathfrak{n}.
	\end{split}
\end{equation}
Here \(C_0\) is well defined  up to numerical equivalence since all fiber curves are numerically equivalent. Then we see that 
\begin{equation}\label{pi*}
	\pi_*C_{0}=0 \text{ and }	\pi_*C_i=l_i \text { for } i=1, \, \ldots, \, \mathfrak{n} .
\end{equation}
\begin{prop}\label{mc_prop}
	The curves $C_0, \cdots, C_{\mathfrak{n}}$ defined in \eqref{mor} span the Mori cone $\overline{\rm{NE}}(\mathbb{P} (\mathcal{E}))$.
\end{prop}
\begin{proof}
	Consider the polyhedral cone $$V:=\Big\{a_0C_0+ \cdots +a_{\mathfrak{n}}C_{\mathfrak{n}} \mid a_i \in \mathbb{R}_{\geq 0} \text{ for } i=0, \ldots, \mathfrak{n}\Big\}.$$ As remarked in \cite[Remark 2.5]{HMP}, a line bundle $L$ on $\mathbb{P} (\mathcal{E})$ is nef if and only if its restriction to every $\mathbb{P}(\mathcal{E}|_{l_i})$ is nef for each $i$. Hence, using the description in (\ref{restnmori}) we have $L$ is a nef line bundle on $\mathbb{P} (\mathcal{E})$ if and only if $L.C_i \geq 0$ for all $i=1,\cdots, \mathfrak{n}$. This implies that $\text{Nef}(\mathbb{P}(\mathcal{E}))=V^{\vee}$. Since \(V\) is polyhedral, we have \(V^{\vee \, \vee} =V \). Thus by duality, we have $\overline{\rm{NE}}(\mathbb{P} (\mathcal{E}))=V$.
\end{proof}

\noindent
Now using the intersection products from \eqref{prodcurve} and \eqref{pushforward}, we have the following:
\begin{equation}\label{intprod}
	\begin{split}
		& \xi \cdot C_{0}= \xi|_{\mathbb{P}(\mathcal{E}|_{l_i} )} \cdot \Sigma_i=\xi_{i} \cdot (\xi_{i}^{r-2}f_i)=1, \text{ and} \\
		& \xi \cdot C_i=\xi|_{\mathbb{P}(\mathcal{E}|_{l_i} )} \cdot \Omega_i=\xi_{i} \cdot (\xi_{i}^{r-1}-m_i \, \xi_{i}^{r-2}f_i)=\mu_{\text{min}}(\mathcal{E}|_{l_i}) \text{ for } i=1, \, \ldots, \, \mathfrak{n}.
	\end{split}
\end{equation}

\section{Seshadri constants of equivariant vector bundles on projective spaces}\label{scpn} 
Let us recall the fan structure of $X:=\mathbb{P}^n$, \(n \geq 2\). Let $\Delta$ denote the fan of \(X\) in the lattice $\Z^n$. Let \(e_1, \, \ldots, \, e_n\) denote the standard basis of \(\Z^n\) and set \(e_0=-e_1- \, \ldots \, -e_n\). Then the fan consists of the rays \(e_0, \, e_1, \, \ldots, \,  e_n\) and maximal cones of the form \(\text{Cone}(e_0, \ldots, \widehat{e}_i, \ldots, e_n)\), where \(i=0, \, 1, \, \ldots, \, n\). Here by \(\widehat{e}_i\) we mean that \(e_i\) is omitted from the relevant collection. The divisors \(D_0, \, D_1, \, \ldots, \, D_n\) corresponding to the rays \(e_0, \, e_1, \, \ldots, \,  e_n\), respectively, are linearly equivalent. Let  the nef cone of $X$ be generated by $D_0$. There are \(\mathfrak{n}\) invariant curves \((\cong \mathbb{P}^1)\) in \(X\), say \(l_1, \, \ldots, \, l_{\mathfrak{n}}\) where \(\mathfrak{n}:={n+1 \choose 2} \). All the invariant curves are numerically equivalent to each other. So the Mori cone is generated by any of them, say \(l_1\).

Let $\mathcal{E}$ be a nef equivariant vector bundle of rank $r$ on $X$. Let 
\begin{equation*}
	\pi:\mathbb{P}(\mathcal{E} )\rightarrow X
\end{equation*}
denote the projectivization map. Let 
\begin{equation*}
	\begin{split}
		&m_j=\text{deg}(\mathcal{E}|_{l_j})-\mu_{\text{min}}(\mathcal{E}|_{l_j}) \text{ for }j=1, \, \ldots, \, \mathfrak{n}.
	\end{split}
\end{equation*}
\noindent
Then from \eqref{mor}, we have the curves \(C_0, \, C_1, \ldots, C_{\mathfrak{n}}\).
As before, we have
\begin{equation*}
	\pi_*C_{0}=0,  \text{ and } \pi_*C_j=l_j \text { for } j=1, \, \ldots, \, \mathfrak{n}.
\end{equation*}
By Proposition \ref{mc_prop}, we have that 
\begin{equation*}
	\overline{\rm{NE}}(\mathbb{P} (\mathcal{E}))= \R_{\geq 0} \, C_0 + \R_{\geq 0}C_1 + \cdots + \R_{\geq 0} \, C_{\mathfrak{n}}.
\end{equation*}
We have the following intersection products (see \eqref{intprod})
\begin{equation}\label{intprodP}
	\xi \cdot C_{0}=1, \text{ and } \xi \cdot C_j=\mu_{\text{min}}(\mathcal{E}|_{l_j}) \text{ for } j=1, \, \ldots, \, \mathfrak{n}.
\end{equation}

\begin{thm}\label{SC_P}
	Let $\mathcal{E}$ be a nef equivariant vector bundle of rank $r$ on the projective space $X=\mathbb{P}^n \, (n \geq 2)$  such that 
	\begin{equation}\label{A1}
		\mu_{\text{min}}(\mathcal{E}|_l)  \leq \, \min\limits_{1 \, \leq \, i \, \leq \mathfrak{n}} \left\{\mu_{\text{min}}(\mathcal{E}|_{l_i})  \right\} \text{  for all lines } l \text{ in } \mathbb{P}^n. 
	\end{equation}
	Then for any point $x \in X$, we have 
	\[\varepsilon(\mathcal{E},x)=\, \min\limits_{1 \, \leq \, i \, \leq \mathfrak{n}} \left\{\mu_{\text{min}}(\mathcal{E}|_{l_i})  \right\}.\]	
\end{thm}

\begin{proof}
	Let \(x \in \mathbb{P}^n\) and \(C \in \mathcal{C}_{\mathcal{E}, x}\). Then \(C \equiv a_0 C_0 + a_1C_1 +\, \ldots \,  + a_{\mathfrak{n}} C_{\mathfrak{n}}\), where \(a_i \in \Z_{\geq 0}\) for \(i=0, \ldots, {\mathfrak{n}}\). Now using \eqref{intprodP}, we have
	\begin{equation}\label{P1}
		\begin{split}
			\frac{\xi \cdot C}{m} & =\frac{a_0 + a_1 \, \mu_{\text{min}}(\mathcal{E}|_{l_1})  + \, \ldots \, + a_{\mathfrak{n}} \, \mu_{\text{min}}(\mathcal{E}|_{l_{\mathfrak{n}}})   }{m}\\
			& \geq  \frac{a_0 + (a_1  + \, \ldots \, + a_{\mathfrak{n}} ) \, \min\limits_{1 \, \leq \, i \, \leq \mathfrak{n}} \left\{\mu_{\text{min}}(\mathcal{E}|_{l_i})  \right\}  }{m}.
		\end{split}
	\end{equation} 
	where \(m={\rm mult}_{x} \pi_*C\).  choose a hyperplane \(H\) through the point \(x\) such that \(\pi_* C \not\subset H\). Then by B\'ezout's theorem, we have
	\begin{equation} \label{projbez1}
		\pi_*C \cdot H \geq m	
	\end{equation}
	Note that \(\pi_* C \equiv (a_1 + \, \ldots \, + a_{\mathfrak{n}} ) \ l_1\) and \(H \equiv D_0\). Then from \eqref{projbez1}, it follows that 
	\[(a_1 + \, \ldots \, + a_{\mathfrak{n}} ) \geq m, \] 
	since $D_0 \cdot l_1=1$. 
	Combining this with \eqref{P1}, we get \[\frac{\xi \cdot C}{m} \geq \, \min\limits_{1 \, \leq \, i \, \leq \mathfrak{n}} \left\{\mu_{\text{min}}(\mathcal{E}|_{l_i})  \right\} .\] Hence we have 
	\begin{equation}\label{proj_ineq2}
		\varepsilon(\mathcal{E},x) \geq \, \min\limits_{1 \, \leq \, i \, \leq \mathfrak{n}} \left\{\mu_{\text{min}}(\mathcal{E}|_{l_i})  \right\} .
	\end{equation}
Note that, we have not used our hypothesis on slope of the restriction of the bundle yet.

	To get the reverse inequality, take any line \(l\) passing through the point \(x\). Then we have 
	\begin{equation}\label{proj_ineq1}
		\begin{split}
			\varepsilon(\mathcal{E},x)  & \leq \varepsilon(\mathcal{E}|_l,x) ~ (\text{by \eqref{FMcorp3.21}})\\
			&=\mu_{\text{min}}(\mathcal{E}|_{l}) ~( \text{by Proposition \ref{FMcor3.21}}),\\
			& \leq \, \min\limits_{1 \, \leq \, i \, \leq \mathfrak{n}} \left\{\mu_{\text{min}}(\mathcal{E}|_{l_i})  \right\}  ~ ( \text{by assumption \eqref{A1}}).
		\end{split}
	\end{equation}
	From \eqref{proj_ineq2} and \eqref{proj_ineq1}, we see that \[\varepsilon(\mathcal{E},x) = \, \min\limits_{1 \, \leq \, i \, \leq \mathfrak{n}} \left\{\mu_{\text{min}}(\mathcal{E}|_{l_i})  \right\} .\] 
\end{proof}

\begin{rmk}{\rm
	We observe from \eqref{proj_ineq2} and the proof of Theorem \ref{SC_P} that, for any arbitrary nef equivariant vector bundle \(\mc{E}\) on \(\mathbb{P}^n\) not necessarily satisfying the hypothesis of Theorem \ref{SC_P}, we always have  
	\begin{equation}\label{proj_ineq}
	\varepsilon(\mathcal{E},x) \geq \, \min\limits_{1 \, \leq \, i \, \leq \mathfrak{n}} \left\{\mu_{\text{min}}(\mathcal{E}|_{l_i})  \right\} .	
	\end{equation}
}
\end{rmk}

\begin{question}{\rm
	Does the equality always hold in \eqref{proj_ineq}?}
	\end{question}

\begin{ex}$($cf. \cite[Example 1.3]{FM}$)$ {\rm  Consider the tangent bundle $\mathscr{T}_{ \mathbb{P}^n}$ over the projective \(n\)-space \(\mathbb{P}^n\) for \(n \geq 2\).  For any line \(l \subset \mathbb{P}^n\) the restriction of the tangent bundle $\mathscr{T}_{ \mathbb{P}^n}$ is given by \[\mathscr{T}_{ \mathbb{P}^n}|_{l}=\mathcal{O}_{\mathbb{P}^1}(2) \oplus \mathcal{O}_{\mathbb{P}^1}(1) \oplus \, \cdots \, \oplus \mathcal{O}_{\mathbb{P}^1}(1)\](see \cite[Example p. 14]{okonek}). Thus $\mathscr{T}_{ \mathbb{P}^n}$ satisfies the condition \eqref{A1}, hence for any \(x \in \mathbb{P}^n\) the Seshadri constant is given by \[\varepsilon(\mathscr{T}_{ \mathbb{P}^n}, x)=1.\]	
}	
\end{ex}

\begin{rmk}{\rm
More generally, the condition \eqref{A1} in Theorem \ref{SC_P} is satisfied by nef equivariant uniform vector bundles on $\mathbb{P}^n$. See \cite[Definition 2.2.1]{okonek} for the definition of a uniform bundle.
}
\end{rmk}

\section{Seshadri constants of equivariant vector bundles on Hirzebruch surfaces}\label{hirz}
 Consider the Hirzebruch surface $X_2$. We follow the notations introduced in Subsection \ref{bott}. The invariant curves in $X_2$ are $D'_1, \, D'_2, \, D_1$ and $D_2$. 
 
 Let $\mathcal{E}$ be a nef equivariant vector bundle of rank $r$ on $X_2$ and
 \begin{equation*}
 \pi:\mathbb{P}(\mathcal{E}) \rightarrow X
 \end{equation*}
denote the projectivization map. Let
\begin{equation*}
\begin{split}
&\mu_j:=\mu_{\text{min}}(\mathcal{E}|_{D_j}) \text{, } \mu'_j:=\mu_{\text{min}}(\mathcal{E}|_{D'_j}) \text{ for }j=1, \,2,
\end{split}
\end{equation*}
and
\begin{equation*}
\begin{split}
&m'_j=\text{deg}(\mathcal{E}|_{D'_j})-\mu'_j \text{ for }j=1, \, 2,\\
&m_j=\text{deg}(\mathcal{E}|_{D_{j-2}})-\mu_{j-2} \text{ for }j=3,\,4.
\end{split}
\end{equation*}
\noindent
From \eqref{mor}, we have the curves \(C_0, \, C_1, \ldots, C_{4}\). 
Then we see that 
\begin{equation}\label{hirz-pushforward}
\pi_*C_0=0, \, \pi_*C_j=D'_j \text { for } j=1, \, 2, \, \pi_*C_j=D_{j-2} \text { for } j=3, \, 4.
\end{equation}
The Mori cone is given by
\begin{equation*}
\overline{\rm{NE}}(\mathbb{P} (\mathcal{E}))= \R_{\geq 0} \, C_0 + \R_{\geq 0}C_1 + \cdots + \R_{\geq 0} \, C_4,
\end{equation*}
(see Proposition \ref{mc_prop}).
We have the following intersection products (see \eqref{intprod})
\begin{equation}\label{hirz-intersection}
	\xi \cdot C_0=1, \, \xi \cdot C_j=\mu'_j \text{ for } j=1, \,2, \text{ and } \xi \cdot C_j=\mu_{j-2} \text{ for } j=3, \,4.
\end{equation}

\begin{thm}\label{SC_H}
	Let $\mathcal{E}$ be an equivariant nef vector bundle of rank \(r\) on the Hirzebruch surface \(X_2\) satisfying the following condition:
	\begin{equation}\label{key_H}
		\mu_{\text{min}}(\mathcal{E}|_{D_1}) =\mu_{\text{min}}(\mathcal{E}|_{D_1'}) 
	\end{equation}
Then for any \(x \in X_2\), the Seshadri constant satisfies the following bounds:
\begin{equation*}
	\begin{split}
		\text{min}\{\mu_1, \,\mu_2, \,\mu'_2\} \, \leq & ~ \varepsilon(\mathcal{E},x) \, \leq \text{min}\{\mu_1, \,\mu'_2\} \  {\rm if}\ x\in \Gamma_2, \\
	\text{min}\{\mu_1, \,\mu_2\} \, \leq & ~ \varepsilon(\mathcal{E},x) \, \leq \mu_1 \  {\rm if}\ x\notin \Gamma_2.
	\end{split}
\end{equation*}
In addition, if \(\mu_2 \geq \mu_1 \), then the Seshadri constant is given by:
\begin{equation*}
	\varepsilon(\mathcal{E},x) =
	\begin{cases}
		\text{min}\{\mu_1,\, \mu'_2\}, & {\rm if}\ x\in \Gamma_2, \\
		\mu_1,  & {\rm if}\ x\notin \Gamma_2.
	\end{cases}
\end{equation*} 
\end{thm}

For the rest of the section we prove the above theorem. Let \(x \in X_2\)	and consider \(C \in \mathcal{C}_{\mathcal{E}, x}\). We can write \(C \equiv a_0 C_0+ \ldots +a_4 C_4\) with \(a_0, \ldots, a_4 \in \Z_{\geq 0}\). Then using \eqref{hirz-intersection}we have 
\begin{equation}\label{H1}
	\frac{\xi \cdot C}{m}= \frac{a_0 + (a_1 + a_3) \mu_1 + a_2 \mu_2'+ a_4 \mu_2}{m},
\end{equation}
where \(m={\rm mult}_{x} \pi_*C\). By \eqref{linequiv} and \eqref{hirz-pushforward}, we have
\begin{equation}\label{hirz-pushforward2}
\pi_*C \equiv (a_2+a_4) D_2'+(a_1+ a_3+c_{1,2} a_4) D_1'.
\end{equation}
Set 
\begin{equation}\label{coefficients}
\begin{split}
&p_1=a_2+a_4 \text{ and } p_2=a_1+ a_3+c_{1,2} a_4.
\end{split}
\end{equation}
Since the fibre curves are numerically equivalent, we also have that 
\begin{equation}\label{hirz-equiv}
	D_1 \equiv  D'_1 \equiv \Gamma^{(2)}_2.
\end{equation}
We first prove the following lemma.

\begin{lemma}\label{H5}
	Let $\mathcal{E}$ be an equivariant nef vector bundle of rank \(r\) on \(X_2\) satisfying \eqref{key_H}. Then we have 
	\[\mu_{\text{min}}(\mathcal{E}|_{\Gamma_2^{(2)}})=\mu_1.\]
\end{lemma}

\begin{proof}
	First note that 
	\begin{equation}\label{H5.1}
		\begin{split}
			\mu_{\text{min}}(\mathcal{E}|_{\Gamma_2^{(2)}})=& \, \varepsilon(\mathcal{E}|_{\Gamma_2^{(2)}},x) \ (\text{by Proposition \ref{FMcor3.21}})\\
			=& \inf\limits_{\substack{C \in \mathcal{C}_{\mathcal{E}, x}\\ C \subseteq \mathbb{P}(\mathcal{E}|_{\Gamma_2^{(2)}})}} \frac{\xi|_{\mathbb{P}(\mathcal{E}|_{\Gamma_2^{(2)}})} \cdot C}{{\rm mult}_{x} \pi_*C}  \\
			=&  \inf\limits_{\substack{C \in \mathcal{C}_{\mathcal{E}, x}\\ C \subseteq \mathbb{P}(\mathcal{E}|_{\Gamma_2^{(2)}})}} \xi \cdot C \ (\text{as } \pi_*C=\Gamma_2^{(2)}).
		\end{split}
	\end{equation}
	
	Let \(C \in  \mathcal{C}_{\mathcal{E}, x}\) and write \(C \, \equiv \, a_0 C_0+ \, \ldots \, + a_4 C_4\) for some nonnegative integers \(a_0, \, \ldots, \, a_4 \). Moreover, if \( C \subseteq \mathbb{P}(\mathcal{E}|_{\Gamma_2^{(2)}})\) then we have 
	\[\pi_*C=\Gamma_2^{(2)} \equiv D_1'.\] Thus from \eqref{hirz-pushforward2}, we get 
	\[a_2+a_4=0 \text{ and } a_1+a_3+ c_{1,2} \, a_4=1.\] 
	Hence \(C\) will be of the following forms
	\begin{equation*}\label{H7}
		C \equiv a_0 C_0 + C_1, \text{ or }C \equiv a_0 C_0 + C_3.
	\end{equation*}
		However, in both of the cases, we have \[\xi \cdot C=a_0+\mu_1\] as \(\mu_1=\mu'_1\) by \eqref{key_H}. Hence from \eqref{H5.1}, we get \(\mu_{\text{min}}(\mathcal{E}|_{\Gamma_2^{(2)}})=\mu_1\). 
\end{proof}
\begin{prop}\label{H2}
	With the notation as in Theorem \ref{SC_H}, for all \(x \in X_2\), we have 
	\begin{equation*}
		\text{min}\{\mu_1, \, \mu_2,\, \mu'_2\} \leq \varepsilon(\mathcal{E},x) \leq \mu_1.
	\end{equation*}
\end{prop}

\begin{proof}
By construction, we have that \(x \in \Gamma_2^{(2)}\). Then, using \eqref{FMcorp3.21} together with Proposition \ref{FMcor3.21}, we have that 
\begin{equation*}\label{hirz-max}
\varepsilon(\mathcal{E},x) \leq \mu_{\text{min}}(\mathcal{E}|_{\Gamma_2^{(2)}}) .
\end{equation*}
Combining this with Lemma \ref{H5}, we get
\begin{equation}\label{hirz-max1}
	\varepsilon(\mathcal{E},x) \leq \mu_1 .
\end{equation}

Now suppose that $\pi_* C \,\not\subset\, \Gamma_2^{(2)}$. Then, using \eqref{hirz-equiv},  we have 
	$$\pi_*C \cdot D_1 \,=\, \pi_* C \cdot \Gamma_2^{(2)} \,\ge\, m (\text{mult}_x\Gamma_2^{(2)}) \,\ge\, m,$$ by B\'ezout's
	theorem. Again, by \eqref{hirz-equiv} together with Proposition \ref{dual-basis}, we have  
	$$\pi_* C \cdot D_1 \,=\, p_1.$$
	So $p_1 \,\ge\, m$. Thus from \eqref{H1}, we get 
	\begin{equation}\label{hirz_ineq1}
		\begin{split}
				\frac{\xi \cdot C}{m} & = \frac{a_0 + (a_1 + a_3) \mu_1 + a_2 \mu_2'+ a_4 \mu_2}{m}\\
				& \geq \frac{a_0 + (a_1 + a_3) \mu_1 +( a_2 + a_4 )\text{min}\{\mu_2,\, \mu'_2\}}{m}\\
				& \geq \text{min}\{\mu_2,\, \mu'_2\}.
		\end{split}
	\end{equation}
Next suppose that $\pi_* C\,\subset\, \Gamma_2^{(2)}$. Then from the definition of 
$\varepsilon(\mathcal{E}|_{\Gamma_2^{(2)}},x)$ it follows that 

\begin{equation}\label{hirz-min}
\frac{\xi\cdot C}{m} \,=\, \frac{\xi|_{\mathbb{P}(\mathcal{E}|_{\Gamma_2^{(2)}})} \cdot C}{m} \,\ge\, \varepsilon(\mathcal{E}|_{\Gamma_2^{(2)}},x)=\mu_1,
\end{equation}
where the last equality follows from Lemma \ref{H5}. Consequently, from \eqref{hirz_ineq1} and \eqref{hirz-min}, we have $$\frac{\xi \cdot C}{m} \,\ge\,
\text{min}\{\mu_1, \, \mu_2,\, \mu'_2\}$$ for all irreducible
and reduced curves $C \,\in\, \mathcal{C}_{\mathcal{E}, x}$. Hence the proposition follows. 
\end{proof}

\begin{proof}[Proof of Theorem \ref{SC_H}] We consider the following two cases.
	
\noindent	
\textbf{ Case 1}: \(x \in \Gamma_2\). 
We have $D'_2 = \Gamma_2$ (see Subsection \ref{qc}). In this case, by \eqref{FMcorp3.21}, we have 
\begin{equation*}
\varepsilon(\mathcal{E},x) \leq \mu'_2.
\end{equation*}
Thus, from \eqref{hirz-max1}, we have that
\begin{equation*}
	\varepsilon(\mathcal{E},x) \leq \text{min}\{\mu_1, \, \mu'_2\}.
\end{equation*}
Combining this with Proposition \ref{H2}, the theorem follows for \(x \in \Gamma_2\).

\noindent
\textbf{ Case 2}: \(x \notin \Gamma_2\).

Note that $x \,\in\, \Gamma_2^{(2)} \,=\, X_2^{(2)}$. Hence 
$x \,\in\,\Gamma_2^{(2)} \setminus \Gamma_2$. Recall from \eqref{hirz-pushforward2} and \eqref{coefficients} that $\pi_* C \,\equiv \, p_1\Gamma_2+p_2\Gamma_2^{(2)} \,\subset\, X_2$ is an irreducible
and reduced curve such that $m\,:=\,{\rm mult}_x\pi_* C \,>\, 0$. Thus we have  
$\pi_*C \,\not\subset\, D_2'$ by \eqref{curveGam}. 
Now, by \eqref{linequiv}, $$D_2' \,\sim_{\text{lin}}\, D_2-c_{1,2}D_1.$$  Hence  
$$0 \,\le\, \pi_*C\cdot D_2' \,=\, \pi_*C\cdot(D_2-c_{1,2}D_1) \,=\, p_2-c_{1,2}p_1\, .$$ So $$p_2
\,\ge\, c_{1,2} \, p_1.$$ 
Plugging the values of \(p_1\) and \(p_2\) from \eqref{coefficients}, we have
\begin{equation}\label{hirz-inequality}
\begin{split}
&a_1+a_3+c_{1,2} \, a_4 \geq c_{1,2} \, (a_2+a_4)\\
&\Rightarrow a_1+a_3 \geq c_{1,2} \, a_2 \geq a_2.
\end{split}
\end{equation}
 The last inequality follows since \(c_{1,2} \geq 1\).

Now suppose that $\pi_*C \,\ne \, \Gamma_2^{(2)} \,=\, X_2^{(2)}$. Then 
B\'ezout's theorem and \eqref{hirz-equiv} together give that
$$m \le \pi_*C \cdot X_2^{(2)} = \pi_*C \cdot D_1 = p_1=a_2+a_4.$$ Thus using \eqref{hirz-inequality}, we have \[a_1+a_3+a_4 \geq m.\] Hence from \eqref{H1}, we obtain 
\begin{equation}\label{H3}
	\frac{\xi \cdot C}{m} \geq \text{min}\{\mu_1, \, \mu_2\}.
\end{equation}
If \(\pi_* C=\Gamma_2^{(2)}\), then from \eqref{hirz-min}, we see that 
\begin{equation}\label{H4}
	\frac{\xi \cdot C}{m} \geq \mu_1 .
\end{equation}
Hence from \eqref{H3} and \eqref{H4}, we get 
\begin{equation*}\label{hirz_ineq2}
		\varepsilon(\mathcal{E},x) \geq \text{min}\{\mu_1, \, \mu_2\}.
\end{equation*}
Thus the proof follows from Proposition \ref{H2}.
\end{proof}

\begin{ex}{\rm 
Consider the tangent bundle \(\mathcal{E}=\mathscr{T}_{ X_2}\)	on the Hirzebruch surface $X_2$. Then the associated filtrations $(E, \{E^{i}(j)\}_{i=1, \ldots, 4; \, j \in \Z})$ are given by:\\
\[ E^{i}(j) = \left\{ \begin{array}
	{r@{\quad \quad}l}
	{\C^2} & j \leqslant 0 \\ 
	\text{Span }( v_{i} ) & j=1 \\
	0 & j > 1
\end{array} \right. ,\]
(see \cite[Examples 2.3 (5)]{Kly}). Let $\sigma_{i,j}$ denote the maximal cone generated by the primitive vectors \(v_i, v_j\) in the fan $\Delta_2$ of \(X_2\). On each maximal cone, using Klyachko's compatibility condition \eqref{KCC}, the associated characters are given by 
\begin{equation*}\label{H6}
	\begin{split}
&		\mathbf{u}(\sigma_{1,2})=\{(1,0), (0,1)\}, \mathbf{u}(\sigma_{2,3})=\{(c_{1,2},1), (-1,0)\}, \\
&		\mathbf{u}(\sigma_{3,4})=\{(-1,0), (-c_{1,2},-1)\},\mathbf{u}(\sigma_{1,4})=\{(1,0), (0,-1)\}.
	\end{split} 
\end{equation*}
Then restriction of $\mathcal{E}$ to the invariant curves are given by (see Section \ref{subsec5})
\begin{equation*}\label{H8}
	\begin{split}
	&	\mathcal{E}|_{D'_1}=\mathcal{O}_{\mathbb{P}^1} \oplus \mathcal{O}_{\mathbb{P}^1} (2), \,	\mathcal{E}|_{D'_2}=\mathcal{O}_{\mathbb{P}^1}(-c_{1,2}) \oplus \mathcal{O}_{\mathbb{P}^1} (2), \\
	& 	\mathcal{E}|_{D_1}=\mathcal{O}_{\mathbb{P}^1} \oplus \mathcal{O}_{\mathbb{P}^1} (2),\, 	\mathcal{E}|_{D_2}=\mathcal{O}_{\mathbb{P}^1}(c_{1,2}) \oplus \mathcal{O}_{\mathbb{P}^1} (2)
	\end{split}
\end{equation*}(see Subsection \ref{subsec5}). By Proposition \ref{HMPT2.1}, $\mathcal{E}$ is not nef. 
Consider \(D=a_1 D_1 + a_2 D_2 \) with \(a_1 \geq c_{1,2}, \, a_2 \geq 0\). Then 
\begin{equation*}
D \cdot D_1'=D \cdot D_1= a_2, \,  D \cdot D'_2= a_1 \text{ and } D \cdot D_2=a_1+ c_{1,2} \, a_2.
\end{equation*}
Thus restrictions of $\mathcal{E}\otimes \mathcal{O}(D)$ to the invariant curves are given by
\begin{equation*}\label{H10}
\begin{split}
	&	(\mathcal{E}\otimes \mathcal{O}(D))|_{D'_1}=\mathcal{O}_{\mathbb{P}^1}(a_2) \oplus \mathcal{O}_{\mathbb{P}^1} (2+a_2), \\
	&	(\mathcal{E}\otimes \mathcal{O}(D))|_{D'_2}=\mathcal{O}_{\mathbb{P}^1}(a_1-c_{1,2}) \oplus \mathcal{O}_{\mathbb{P}^1} (2+a_1), \\
	& 	(\mathcal{E}\otimes \mathcal{O}(D))|_{D_1}=\mathcal{O}_{\mathbb{P}^1}(a_2) \oplus \mathcal{O}_{\mathbb{P}^1} (2+a_2), \\
	&	(\mathcal{E}\otimes \mathcal{O}(D))|_{D_2}=\mathcal{O}_{\mathbb{P}^1}(a_1+c_{1,2} \, a_2+c_{1,2}) \oplus \mathcal{O}_{\mathbb{P}^1} (a_1+c_{1,2} \, a_2+2).
\end{split}
\end{equation*}
This shows that \(\mathcal{E}\otimes \mathcal{O}(D)\) is nef (by Proposition \ref{HMPT2.1}). We also have that
\begin{equation*}
	\mu_1'=a_2=\mu_1, \, \mu_2'=a_1-c_{1,2}, \mu_2 \geq a_1+c_{1,2} \, a_2.   
\end{equation*}
So \(\mathcal{E}\otimes \mathcal{O}(D)\) satisfies the conditions of Theorem \ref{SC_H}, hence for \(x \in X_2\) the Seshadri constant is given by 
\begin{equation*}
	\varepsilon(\mathcal{E}\otimes \mathcal{O}(D),x) =
	\begin{cases}
		\text{min}\{a_1-c_{1,2},\, a_2\}, & {\rm if}\ x\in \Gamma_2, \\
		a_2,  & {\rm if}\ x\notin \Gamma_2.
	\end{cases}
\end{equation*} 
Note that the discriminant of \(\mathcal{E}\otimes \mathcal{O}(D)\) is non-zero. Also for \(c_{1,2} \geq 2\), this bundle is unstable with respect to the anticanonical divisor and it is stable when \(c_{1,2}=1\) (see \cite[Corollary 4.2.7]{DDK}).
}

\end{ex}

\begin{ex}\label{hirz-ex}{\rm  
	Consider the vector space \(E=\C^2\) and three distinct one dimensional subspaces \(L_1, L_2\) and \(L_{3}\) in \(E \). Now define the filtrations \( \left( E, \{ E^{v_j}(i) \}_{j=1, 2, 3} \right) \)  as follows:
	
	$
	E^{v_j}(i) = \left\{ \begin{array}{ccc}
		
		E& i \leqslant 0 \\
		
		L_j & i =1 \\ 
		
		0 & i > 1
	\end{array} \right. 
	$
	for \(j=1, 2, 3\) 
	and 
	$
	E^{v_4}(i) = \left\{ \begin{array}{ccc}
		
		E &  i \leq 0 \\
		
		0 &  i > 0
	\end{array} \right.
	$
.
	
	Hence, the filtrations \( \left( E, \{ E^{v_j}(i) \}_{j=1, \ldots, 4} \right) \) correspond to a rank \(2\) equivariant indecomposable vector bundle on \(X_2\), say \(\mathcal{E}\) (see \cite[Proposition 6.1.1]{DDK}). 
	
	As before, let $\sigma_{i,j}$ denote the maximal cone generated by the primitive vectors \(v_i, v_j\) in the fan $\Delta_2$ of \(X_2\). On each maximal cone, using Klyachko's compatibility condition \eqref{KCC}, the associated characters are given by 
	\begin{equation*}\label{H6.1}
		\begin{split}
			&		\mathbf{u}(\sigma_{1,2})=\{(1,0), (0,1)\}, \mathbf{u}(\sigma_{2,3})=\{(c_{1,2},1), (-1,0)\}, \\
			&		\mathbf{u}(\sigma_{3,4})=\{(-1,0), (0,0)\},\mathbf{u}(\sigma_{1,4})=\{(1,0), (0,0)\}.
		\end{split} 
	\end{equation*}
Then following Subsection \ref{subsec5}, restrictions of $\mathcal{E}$ to the invariant curves are given by
	\begin{equation*}\label{H8.1}
		\begin{split}
			&	\mathcal{E}|_{D'_1}=\mathcal{O}_{\mathbb{P}^1} \oplus \mathcal{O}_{\mathbb{P}^1} (1), 	\mathcal{E}|_{D'_2}=\mathcal{O}_{\mathbb{P}^1}(-c_{1,2}) \oplus \mathcal{O}_{\mathbb{P}^1} (2), \\
			& 	\mathcal{E}|_{D_1}=\mathcal{O}_{\mathbb{P}^1} \oplus \mathcal{O}_{\mathbb{P}^1} (1), 	\mathcal{E}|_{D_2}=\mathcal{O}_{\mathbb{P}^1} \oplus \mathcal{O}_{\mathbb{P}^1} (2).
		\end{split}
	\end{equation*}
	As before, $\mathcal{E}$ is not nef (by Proposition \ref{HMPT2.1}). Consider \(D=a_1 D_1 + a_2 D_2 \) with \(a_1 \geq c_{1,2}, \, a_2 \geq 0\). Then 
	\begin{equation*}
		D \cdot D_1'=D \cdot D_1= a_2, \,  D \cdot D'_2= a_1 \text{ and } D \cdot D_2=a_1+ c_{1,2} \, a_2.
	\end{equation*}
	Thus restrictions of $\mathcal{E}\otimes \mathcal{O}(D)$ to the invariant curves are given by
	\begin{equation*}\label{H10}
		\begin{split}
			&	(\mathcal{E}\otimes \mathcal{O}(D))|_{D'_1}=\mathcal{O}_{\mathbb{P}^1}(a_2) \oplus \mathcal{O}_{\mathbb{P}^1} (1+a_2), \\
			&	(\mathcal{E}\otimes \mathcal{O}(D))|_{D'_2}=\mathcal{O}_{\mathbb{P}^1}(a_1+2) \oplus \mathcal{O}_{\mathbb{P}^1} (a_1-c_{1,2}), \\
			& 	(\mathcal{E}\otimes \mathcal{O}(D))|_{D_1}=\mathcal{O}_{\mathbb{P}^1}(a_2) \oplus \mathcal{O}_{\mathbb{P}^1} (1+a_2), \\
			&	(\mathcal{E}\otimes \mathcal{O}(D))|_{D_2}=\mathcal{O}_{\mathbb{P}^1}(a_1+c_{1,2} \, a_2) \oplus \mathcal{O}_{\mathbb{P}^1} (a_1+c_{1,2} \, a_2+2).
		\end{split}
	\end{equation*}
	This implies that \(\mathcal{E}\otimes \mathcal{O}(D)\) is nef (by Proposition \ref{HMPT2.1}). Also, note that
	\begin{equation*}
		\mu_1'=a_2=\mu_1, \, \mu_2'=a_1-c_{1,2}, \mu_2 = a_1+c_{1,2} \, a_2.   
	\end{equation*}
	So \(\mathcal{E}\otimes \mathcal{O}(D)\) satisfies the conditions of Theorem \ref{SC_H}. Hence for \(x \in X_2\), the Seshadri constant is given by 
	\begin{equation*}
		\varepsilon(\mathcal{E}\otimes \mathcal{O}(D),x) =
		\begin{cases}
			\text{min}\{a_1-c_{1,2},\, a_2\}, & {\rm if}\ x\in \Gamma_2, \\
			a_2,  & {\rm if}\ x\notin \Gamma_2.
		\end{cases}
	\end{equation*} 
Note that for a suitable choice of polarization on \(X_2\), the vector bundle \(\mathcal{E}\) becomes stable (see \cite[Proposition 6.1.5]{DDK}).  If discriminant of $\mathcal{E}$ was zero then by \cite[Theorem 2.5]{bruzzo}, we can conclude that $(\mathcal{E}\otimes \mathcal{O}(D))|_{D'_1}$ is semistable, which is not the case. So discriminant of $\mathcal{E}$ is non-zero.

}

\end{ex}

\begin{rmk}{\rm
We would like to mention that there are examples of equivariant ample vector bundles which do not satisfy the assumption \ref{key_H}. One such example can be found in \cite[Example 5.5]{DJS}. 	}
\end{rmk}
\section{Seshadri constants of equivariant vector bundles on \(X_3\)}\label{X3}

Consider the third stage of Bott tower \(X_3\). There are 12 invariant curves in \(X_3\) which are listed below:
\begin{equation} \label{inv_curves 3}
	\begin{split}
		 & l_1:=D_1' \cap D_2',\, l_2:=D_1' \cap D_3',\, l_3:=D_1' \cap D_2,\, l_4:=D_1' \cap D_3, \\
		 & l_5:=D_2' \cap D_3',\, l_6:=D_2' \cap D_1,\, l_7:=D_2' \cap D_3,\, l_8:=D_3' \cap D_1, \\
		 & l_9:=D_3' \cap D_2,\, l_{10}:=D_1 \cap D_2,\, l_{11}:=D_1 \cap D_3,\, l_{12}:=D_2 \cap D_3.
	\end{split}
\end{equation}
Recall from Proposition \ref{dual-basis} that a basis of the Mori cone $\overline{\rm{NE}}(X_3)$ is given by the curves $\Gamma_3,\, \Gamma_3^{(2)},\,\Gamma_3^{(3)}$. The invariant curves listed in \eqref{inv_curves 3} can be written in terms of the generators of the Mori cone as follows:
\begin{equation}\label{X lin_equiv}
	\begin{split}
		& l_1\,\equiv\, l_3\,\equiv\,l_6\,\equiv\, l_{10}\,\equiv\,\Gamma_3^{(3)},\\
		& l_2\,\equiv\, l_8\,\equiv\,\Gamma_3^{(2)},\\
		& l_4\,\equiv\, l_{11}\,\equiv \, c_{2,3} \, \Gamma_3^{(2)},\\
		& l_5 \,\equiv \, \Gamma_3, \,\\
		&  l_7  \,\equiv \, \Gamma_3 + c_{1,3} \, \Gamma_3^{(3)}, \\
		& l_9 \,\equiv \, \Gamma_3 + c_{1,2} \, \Gamma_3^{(2)}, \\
		& l_{12}  \,\equiv \, \Gamma_3 + c_{1,2} \, \Gamma_3^{(2)}+(c_{1,3} + c_{1,2} \, c_{2, 3} )\, \Gamma_3^{(3)}. 
	\end{split}
\end{equation}

Let $\mathcal{E}$ be an equivariant nef vector bundle of rank $r$ on $X_3$ and
\begin{equation*}
	\pi:\mathbb{P}(\mathcal{E}) \rightarrow X_3
\end{equation*}
denote the projectivization map. Let
\begin{equation*}
		\mu_j:=\mu_{\text{min}}(\mathcal{E}|_{l_j}) \text{ and } m_j=\text{deg}(\mathcal{E}|_{l_j})-\mu_j \text{ for }j=1, \, \ldots, \, 12.
\end{equation*}
\noindent
Then from \eqref{mor}, we have the curves \(C_0, \, C_1, \ldots, C_{12}\). Recall from \eqref{pi*} that 
\begin{equation}\label{3-pushforward}
	\pi_*C_0=0 \text{ and } \pi_*C_j=l_j \text { for } j=1, \, \ldots, \, 12.
\end{equation}
By Proposition \ref{mc_prop}, we have that 
\begin{equation*}
	\overline{\rm{NE}}(\mathbb{P} (\mathcal{E}))= \R_{\geq 0} \, C_0 + \R_{\geq 0}C_1 + \cdots + \R_{\geq 0} \, C_{12}.
\end{equation*}
We have the following intersection products (see \eqref{intprod})
\begin{equation}\label{3-intersection}
	\xi \cdot C_0=1, \, \xi \cdot C_j= \mu_{\text{min}}(\mathcal{E}|_{l_j}) \text{ for } j=1, \,\ldots, \, 12.
\end{equation}

\begin{thm}\label{SC_3}
	Let $\mathcal{E}$ be an equivariant nef vector bundle of rank \(r\) on \(X_3\) satisfying the following conditions:
	\begin{equation}\label{key condn}
		\begin{split}
		&	\mu_1=\mu_3=\mu_6=\mu_{10}; \, \mu_2= \mu_8,\\
		& \mu_7, \, \mu_9, \, \mu_{12} \geq \mu_5 ,
		\end{split}
	\end{equation}
	\begin{equation}\label{key condn1}
		\begin{split}
		&	\mu_4, \,  \mu_{11} \geq c_{2,3} \, \mu_8, \\
		&  \mu_9, \,  \mu_{12} \geq c_{1,2} \, \mu_8,
		\end{split}
	\end{equation}
	\begin{equation}\label{key condn2}
	\begin{split}
	&	\mu_7 \geq c_{1,3} \, \mu_{10}, \\
	& \mu_{12} \geq (c_{1,3}+ c_{1,2} \, c_{2,3}) \, \mu_{10}.
	\end{split}
\end{equation}
	Then the Seshadri constants of $\mathcal{E}$ at any \(x \in X_3\) are given by the following. 
	\begin{equation*}
		\varepsilon(X_3,\mathcal{E},x) =
		\begin{cases}
			{\rm min}\{\mu_{\text{min}}(\mathcal{E}|_{\Gamma_3}),\varepsilon(X_3^{(2)},\mathcal{E}|_{X_3^{(2)}},x)\}, &
			{\rm if}\ x\in \Gamma_3, \\
			\varepsilon(X_3^{(2)},\mathcal{E}|_{X_3^{(2)}},x),  & {\rm if}\ x\notin \Gamma_3.
		\end{cases}
	\end{equation*} 
\end{thm}

In order to prove Theorem \ref{SC_3}, we will prove a series of results.

Let $\mathcal{E}$ be an equivariant nef vector bundle of rank \(r\) on \(X_3\) satisfying the conditions of Theorem \ref{SC_3}. Let \(x \in X_3\)	and \(C \in \mathcal{C}_{\mathcal{E}, x}\) and denote \(m={\rm mult}_{x} \pi_*C\). We can write \(C \equiv a_0 C_0+ \ldots +a_{12} C_{12}\) with \(a_0, \ldots, a_{12} \in \Z_{\geq 0}\). Then using \eqref{3-intersection} we have 
\begin{equation}\label{X1}
	\frac{\xi \cdot C}{m}= \frac{a_0 + \sum_{j=1}^{12} a_j \, \mu_j}{m}.
\end{equation}
By \eqref{X lin_equiv} and \eqref{3-pushforward}, we have
\begin{equation}\label{X-pushforward2}
	\pi_*C \equiv p_1 \, l_5 + p_2 \, l_8 + p_3 \, l_{10}, 
\end{equation}
where
\begin{equation}\label{X-coefficients}
	\begin{split}
		&p_1=a_5+a_7+ a_9 +a_{12}, \\
		&p_2=a_2+ a_8+ (a_9+a_{12}) c_{1,2}+(a_4+a_{11}) c_{2,3} \text{ and}\\
		& p_3=a_1+a_3+a_6+a_{10}+ (a_7+a_{12}) c_{1,3} + a_{12} c_{1,2} c_{2,3}.
	\end{split}
\end{equation}
From the quotient construction (see Subsection \ref{qc}), we have $\Gamma_3 \, =l_5$ and hence \(\mu_{\text{min}}(\mathcal{E}|_{\Gamma_3})=\mu_5\). We then prove the following. 

\begin{lemma}\label{SC3lem1}
	With the notations as above, we have \[\mu_{\text{min}}(\mathcal{E}|_{\Gamma_3^{(2)}})=\mu_8.\]
\end{lemma}

\begin{proof}
Using Proposition \ref{FMcor3.21}, we have 
	\begin{equation}\label{SC3_1}
		\begin{split}
			\mu_{\text{min}}(\mathcal{E}|_{\Gamma_3^{(2)}})=& \, \varepsilon(\mathcal{E}|_{\Gamma_3^{(2)}},x) \\
			=& \inf\limits_{\substack{C \in \mathcal{C}_{\mathcal{E}, x}\\ C \subseteq \mathbb{P}(\mathcal{E}|_{\Gamma_3^{(2)}})}} \frac{\xi|_{\mathbb{P}(\mathcal{E}|_{\Gamma_3^{(2)}})} \cdot C}{{\rm mult}_{x} \pi_*C}  \\
			=&  \inf\limits_{\substack{C \in \mathcal{C}_{\mathcal{E}, x}\\ C \subseteq \mathbb{P}(\mathcal{E}|_{\Gamma_3^{(2)}})}} \xi \cdot C \ (\text{as } \pi_*C=\Gamma_3^{(2)}).
		\end{split}
	\end{equation}
	
	Let \(C \in  \mathcal{C}_{\mathcal{E}, x}\) and write \(C \, \equiv \, a_0 C_0+ \, \ldots \, + a_{12} C_{12}\) for some nonnegative integers \(a_0, \, \ldots, \, a_{12} \). Moreover, if \( C \subseteq \mathbb{P}(\mathcal{E}|_{\Gamma_3^{(2)}})\) then we have 
	\[\pi_*C=\Gamma_3^{(2)} \equiv l_8.\] Thus from \eqref{X-pushforward2}, we get 
	\[p_1=p_3=0 \text{ and } p_2=1.\] 
	This implies,
	\[a_5=a_7=a_9=a_{12}=0=a_1=a_3=a_6=a_7=a_{10}, \, a_2+a_8+(a_4+a_{11}) \, c_{2,3}=1.\]
	If \(c_{2,3} > 1\), we have two possibilities, namely
	\begin{equation*}\label{SC3_7}
		C \equiv a_0 C_0 + C_2, \text{ or }C \equiv a_0 C_0 + C_8.
	\end{equation*}
If \(c_{2,3} = 1\), we have 
\begin{equation*}
	C \equiv a_0 C_0 + C_2, \text{ or }C \equiv a_0 C_0 + C_8 \text{ or }C \equiv a_0 C_0 + C_4 \text{ or }C \equiv a_0 C_0 + C_{11}.
\end{equation*}
Hence from \eqref{SC3_1}, we get \(\mu_{\text{min}}(\mathcal{E}|_{\Gamma_3^{(2)}})=\mu_8\) as \(\mu_2=\mu_8\) by \eqref{key condn} and $\mu_4, \,  \mu_{11} \geq c_{2,3} \, \mu_8$ by \eqref{key condn1}. 
\end{proof}

\begin{lemma}\label{SC3lem2}
	With the notations as above, we have \[\mu_{\text{min}}(\mathcal{E}|_{\Gamma_3^{(3)}})=\mu_{10}.\]
\end{lemma}

\begin{proof}
	First note that 
	\begin{equation}\label{SC3_2}
		\begin{split}
			\mu_{\text{min}}(\mathcal{E}|_{\Gamma_3^{(3)}})=& \, \varepsilon(\mathcal{E}|_{\Gamma_3^{(3)}},x) \ (\text{by Proposition \ref{FMcor3.21}})\\
			=& \inf\limits_{\substack{C \in \mathcal{C}_{\mathcal{E}, x}\\ C \subseteq \mathbb{P}(\mathcal{E}|_{\Gamma_3^{(3)}})}} \frac{\xi|_{\mathbb{P}(\mathcal{E}|_{\Gamma_3^{(3)}})} \cdot C}{{\rm mult}_{x} \pi_*C}  \\
			=&  \inf\limits_{\substack{C \in \mathcal{C}_{\mathcal{E}, x}\\ C \subseteq \mathbb{P}(\mathcal{E}|_{\Gamma_3^{(3)}})}} \xi \cdot C \ (\text{as } \pi_*C=\Gamma_3^{(3)}).
		\end{split}
	\end{equation}
	
	Let \(C \in  \mathcal{C}_{\mathcal{E}, x}\) and write \(C \, \equiv \, a_0 C_0+ \, \ldots \, + a_{12} C_{12}\) for some nonnegative integers \(a_0, \, \ldots, \, a_{12} \). Moreover, if \( C \subseteq \mathbb{P}(\mathcal{E}|_{\Gamma_3^{(3)}})\) then we have 
	\[\pi_*C=\Gamma_3^{(3)} \equiv l_{10}.\] Now from \eqref{X-pushforward2}, we get 
	\[p_1=p_2=0 \text{ and } p_3=1.\] 
	This implies,
	\[a_5=a_7=a_9=a_{12}=0=a_2=a_4=a_8=a_{11}, \, a_1+a_3+a_6+a_{10}=1.\]
	Thus \(C\) will be of the following forms
	\begin{equation*}\label{SC3_8}
		C \equiv a_0 C_0 + C_1, \text{ or } C \equiv a_0 C_0 + C_3 \text{ or } C \equiv a_0 C_0 + C_6 \text{ or } C \equiv a_0 C_0 + C_{10} .
	\end{equation*}
	However, in all the cases, we have \[\xi \cdot C=a_0+\mu_{10}\] as \(\mu_1=\mu_3=\mu_6=\mu_{10}\) by \eqref{key condn}. Hence from \eqref{SC3_2}, we get \(\mu_{\text{min}}(\mathcal{E}|_{\Gamma_3^{(3)}})=\mu_{10}\). 
\end{proof}

\begin{prop}\label{SC X_3 p1}
	With the notations as in Theorem \ref{SC_3}, we have 
	\begin{equation*}
		\varepsilon(X_3,\mathcal{E},x) \geq {\rm min}\{\mu_{\text{min}}(\mathcal{E}|_{\Gamma_3}),\varepsilon(X_3^{(2)},\mathcal{E}|_{X_3^{(2)}},x)\}
	\end{equation*}
for all \(x \in X_3 \).
\end{prop}

\begin{proof}
	First suppose that $\pi_* C \,\not\subset \, X_3^{(2)}$. Then, using \eqref{X lin_equiv}, \eqref{3-pushforward} and \eqref{X-pushforward2},  we have 
	$$\pi_* C \cdot D_1 \,=\, \pi_*C \cdot X_3^{(2)} \,\ge\, m (\text{mult}_xX_3^{(2)}) \,\ge\, m,$$ by B\'ezout's
	theorem. Again, by \eqref{X lin_equiv}, \eqref{3-pushforward} and \eqref{X-pushforward2} we have  
	$$\pi_* C \cdot D_1 \,=\, p_1.$$
	So $p_1 \,\ge\, m$. Thus from \eqref{X1}, we get 
	\begin{equation*}
		\begin{split}
			\frac{\xi \cdot C}{m} & = \frac{a_0 + \sum_{j=1}^{12} a_j \, \mu_j}{m}\\
			& \geq \mu_5=\mu_{\text{min}}(\mathcal{E}|_{\Gamma_3}) \ (\text{using } p_1 \geq m \text{ and } \eqref{key condn}).
		\end{split}
	\end{equation*}
	Next suppose that $\pi_* C\,\subset\, X_3^{(2)}$. Then from the definition of 
	$\varepsilon(X_3^{(2)}, \mathcal{E}|_{X_3^{(2)}},x)$ it follows that 
	
	\begin{equation}\label{X-min}
		\frac{\xi\cdot C}{m} \,=\, \frac{\xi|_{\mathbb{P}(\mathcal{E}|_{X_3^{(2)}})} \cdot C}{m} \,\ge\, \varepsilon(X_3^{(2)}, \mathcal{E}|_{X_3^{(2)}},x).
	\end{equation}
	
	Consequently, $\frac{\xi \cdot C}{m} \,\ge\, \text{min}\{\mu_{\text{min}}(\mathcal{E}|_{\Gamma_3}),\varepsilon(X_3^{(2)},\mathcal{E}|_{X_3^{(2)}},x)\}$ for all irreducible
	and reduced curves $C \,\in\, \mathcal{C}_{\mathcal{E}, x}$. Hence the proposition follows. 
\end{proof}

\begin{prop} \label{SC X3 p2}
	If \(x \in \Gamma_3\), then the Seshadri constant is given by 
	\begin{equation*}
	\varepsilon(X_3,\mathcal{E},x)=\text{min}	\{\mu_{\text{min}}(\mathcal{E}|_{\Gamma_3}),\varepsilon(X_3^{(2)},\mathcal{E}|_{X_3^{(2)}},x)\}.
	\end{equation*}
\end{prop}

\begin{proof}
	Since \(x \in \Gamma_3 \, (=l_5)\) by \eqref{FMcorp3.21}, we have that 
	\begin{equation*}\label{X-max}
	\begin{split}
			\varepsilon(X_3, \mathcal{E},x) \leq \varepsilon(\Gamma_3, \mathcal{E}|_{\Gamma_3},x)= \mu_{\text{min}}(\mathcal{E}|_{\Gamma_3}).
	\end{split}
	\end{equation*}
Here the last equality follows from Proposition \ref{FMcor3.21}. On the other hand, we always have
\begin{equation}\label{always}
\varepsilon(X_3, \mathcal{E}, x)= \, \inf\limits_{\substack{C \in \mathcal{C}_{\mathcal{E}, x}}} \frac{\xi\cdot C}{{\rm mult}_{x} \pi_*C} \leq \, \inf\limits_{\substack{C \in \mathcal{C}_{\mathcal{E}, x}\\ C \subseteq \mathbb{P}(\mathcal{E}|_{X_3^{(2)}})}} \frac{\xi\cdot
	C}{{\rm mult}_{x} \pi_*C}=\varepsilon(X_3^{(2)},\mathcal{E}|_{X_3^{(2)}},x).
\end{equation}
So \(\varepsilon(X_3,\mathcal{E},x)\leq \text{min}	\{\mu_{\text{min}}(\mathcal{E}|_{\Gamma_3}),\varepsilon(X_3^{(2)},\mathcal{E}|_{X_3^{(2)}},x)\}\) and the proposition follows from Proposition \ref{SC X_3 p1}.
\end{proof}

\begin{lemma} \label{SC X3 lem}
	Suppose that \(\frac{\xi \cdot C}{m} \geq \varepsilon(X_3^{(2)},\mathcal{E}|_{X_3^{(2)}},x)\) for all curves \(C \in \mathcal{C}_{\mathcal{E}, x}\) such that \(\pi_* C \not\subset  X_3^{(2)}\), then 
	\begin{equation*}
	\varepsilon(X_3,\mathcal{E},x)=	\varepsilon(X_3^{(2)},\mathcal{E}|_{X_3^{(2)}},x).
	\end{equation*}
\end{lemma}

\begin{proof}
	Note that by \eqref{always}, we have \(\varepsilon(X_3,\mathcal{E},x) \leq	\varepsilon(X_3^{(2)},\mathcal{E}|_{X_3^{(2)}},x)\). Recall from
	\eqref{X-min} that  for $\pi_* C\,\subset\, X_3^{(2)}$, we have \(	\frac{\xi\cdot C}{m} \, \ge\, \varepsilon(X_3^{(2)}, \mathcal{E}|_{X_3^{(2)}},x).\) Finally, using the hypothesis we have \(\varepsilon(X_3,\mathcal{E},x) \geq \varepsilon(X_3^{(2)},\mathcal{E}|_{X_3^{(2)}},x)\). Hence the lemma follows.
\end{proof}

\begin{prop} \label{SC X3 p3}
	Let \(x \in X_3\) be such that \(x \notin \Gamma_3\). Then 
	\begin{equation*}
			\varepsilon(X_3,\mathcal{E},x)=	\varepsilon(X_3^{(2)},\mathcal{E}|_{X_3^{(2)}},x).
	\end{equation*}
\end{prop}

\begin{proof}
	We consider the following cases.
	
	\noindent
	\textbf{ Case 1}: Let \(x \in \Gamma_3^{(2)}\).
	Note that $x \,\in\,\Gamma_3^{(2)} \setminus \Gamma_3$. Recall from \eqref{X-pushforward2} that $\pi_* C \,=\, p_1 l_5+p_2 l_8 +p_3 l_{10} \,\subset\, X_3$ is an irreducible 	and reduced curve such that $m\,:=\,{\rm mult}_x\pi_* C \,>\, 0$. Thus we have 	$\pi_* C \,\not\subset\, D_2'$ by \eqref{curveGam}. Now, by \eqref{linequiv}, $$D_2' \,\sim_{\text{lin}}\, D_2-c_{1,2}D_1.$$  Hence  
	$$0 \,\le\, \pi_* C\cdot D_2' \,=\, \pi_* C\cdot(D_2-c_{1,2}D_1) \,=\, p_2-c_{1,2}p_1\, .$$ So $$p_2
	\,\ge\, c_{1,2} \, p_1.$$
	From now on suppose that $\pi_* C \,\not\subset \, X_3^{(2)}$. Then arguing as in the proof of Proposition \ref{SC X_3 p1}, we get \(p_1 \geq m\). We next show that 
	\begin{equation*}
		\frac{\xi \cdot C}{m} \geq \mu_8.
	\end{equation*}
From \eqref{X1}, we get 
\begin{equation*}
	\begin{split}
	&	\frac{\xi \cdot C}{m}  = \frac{a_0 + \sum_{j=1}^{12} a_j \, \mu_j}{m}\\
		& \geq \frac{a_0 + a_1 \, \mu_1 + a_3 \, \mu_3+ a_5 \, \mu_5 + a_6 \, \mu_6 + a_7 \, \mu_7 +p_2 \, \mu_8 +a_{10} \, \mu_{10}}{m} \ (\text{using \eqref{key condn1} and \eqref{X-coefficients}})\\
		& \geq \mu_8 \ (\text{using the fact that } p_2 \geq m).
	\end{split}
\end{equation*} 
Finally, from Lemma \ref{SC3lem1} note that \[\mu_8=\mu_{\text{min}}(\mathcal{E}|_{\Gamma_3^{(2)}})=\varepsilon(\Gamma_3^{(2)}, \mathcal{E}|_{\Gamma_3^{(2)}}, x) \geq \varepsilon(X_3^{(2)},\mathcal{E}|_{X_3^{(2)}},x).\]
Hence in this case the proposition follows from Lemma \ref{SC X3 lem}.

\noindent
\textbf{ Case 2}: Let \(x \notin \Gamma_3^{(2)}\). We first show that 
\begin{equation*}\label{mu_10}
	\varepsilon(X_3,\mathcal{E},x)=\mu_{10}.
\end{equation*}
Write \(x=[z_1^0:w_1^0 : z_{2}^0:w_{2}^0:z_3^0:w_3^0]\) and recall from Subsection \ref{qc} that 
\begin{equation*}
	X_3^{(2)}\,=\, \left\{[z_1^0:w_1^0 : z_{2}:w_{2}:z_3:w_3] \,\mid\, (z_i,w_i) 	\,\in\, \C^2 \setminus 0\ \text{ for }  i=2, 3 \right\},
\end{equation*}
and there is an isomorphism
\begin{equation}\label{iso5.1}
	\begin{split}
	X_3^{(2)} \stackrel{\cong } \longrightarrow X_2, \, 	[z_1^0:w_1^0 : z_{2}:w_{2}:z_3:w_3]  \longmapsto [ z_{2}:w_{2}:z_3:w_3]. 
	\end{split}
\end{equation}
Here \(X_2\) is a Hirzebruch surface with Bott number \(c_{2,3}\) and we denote the coordinates by \([ z'_{2}:w'_{2}:z'_3:w'_3]\). The action of the torus (\(\cong (\C^*)^3\)) on \(	X_3^{(2)}\) is given by
\begin{equation*}
\begin{split}
		&(t_1, t_2, t_3) \cdot [z_1^0:w_1^0 : z_{2}:w_{2}:z_3:w_3] \\
	&	=[t_1 \, z_1^0:t_1 \, w_1^0 :t_1^{-c_{1,2}} \, t_2 \, z_{2}:t_2 \, w_{2}:t_1^{-c_{1,3}} \, t_2^{-c_{2,3}} \, t_3 \, z_3:t_3 \, w_3], 
\end{split}
\end{equation*}
and the torus on \(X_2\) is identified with
\[\{(1, t_2, t_3) ~ : ~ t_2, t_3 \in \C^*\} ~ ( \cong (\C^*)^2),\]
so that the isomorphism \eqref{iso5.1} becomes isomorphism of toric varieties.

Let \(D_1^{(2)}:=\mathbb{V}(w'_2) \subset X_3^{(2)}\) and \(D_1'^{(2)}:=\mathbb{V}(z'_2) \subset X_3^{(2)}\). Note that \(D_1^{(2)}=D_2 \cap X_3^{(2)}\) and \(D_1'^{(2)}=D_2' \cap X_3^{(2)}\) and we know that \(D_1^{(2)} \equiv D_1'^{(2)}\). Also observe that (using \cite[Proposition 2.2]{Sc_bott})
\begin{equation*}
	D_1^{(2)} \equiv D_2 \cdot D_1 =l_{10} \text{ and } D_1'^{(2)} \equiv D'_2 \cdot D_1 =l_{6}. 
\end{equation*}
Now we show that 
\begin{equation*}\label{eglast}
	\mu_{\text{min}}(\mathcal{E}|_{	D_1^{(2)}})=\mu_{\text{min}}(\mathcal{E}|_{D_1'^{(2)}})=\mu_{10}.
\end{equation*}

We see that 
\begin{equation}\label{SC3_2.1}
	\begin{split}
		\mu_{\text{min}}(\mathcal{E}|_{D_1^{(2)}})=& \, \varepsilon(\mathcal{E}|_{D_1^{(2)}},x) \ (\text{by Proposition \ref{FMcor3.21})}\\
		=& \inf\limits_{\substack{C \in \mathcal{C}_{\mathcal{E}, x}\\ C \subseteq \mathbb{P}(\mathcal{E}|_{D_1^{(2)}})}} \frac{\xi|_{\mathbb{P}(\mathcal{E}|_{D_1^{(2)}})} \cdot C}{{\rm mult}_{x}\pi_*C}  \\
		=&  \inf\limits_{\substack{C \in \mathcal{C}_{\mathcal{E}, x}\\ C \subseteq \mathbb{P}(\mathcal{E}|_{D_1^{(2)}})}} \xi \cdot C \ (\text{as } \pi_*C=D_1^{(2)}).
	\end{split}
\end{equation}

Let \(C \in  \mathcal{C}_{\mathcal{E}, x}\) and write \(C \, \equiv \, a_0 C_0+ \, \ldots \, + a_{12} C_{12}\) for some nonnegative integers \(a_0, \, \ldots, \, a_{12} \). Moreover, if \( C \subseteq \mathbb{P}(\mathcal{E}|_{D_1^{(2)}})\) then we have 
\[\pi_*C=D_1^{(2)}\equiv l_{10}.\] Thus from \eqref{X-pushforward2}, we get 
\[p_1=p_2=0 \text{ and } p_3=1.\] 
This implies that,
\[a_5=a_7=a_9=a_{12}=0=a_2=a_4=a_8=a_{11}, \, a_1+a_3+a_6+a_{10}=1.\]
So \(C\) will be of the following forms
\begin{equation*}\label{SC3_8.1}
	C \equiv a_0 C_0 + C_1, \text{ or } C \equiv a_0 C_0 + C_3 \text{ or } C \equiv a_0 C_0 + C_6 \text{ or } C \equiv a_0 C_0 + C_{10} .
\end{equation*}
However, in all the cases, we have \[\xi \cdot C=a_0+\mu_{10}\] as \(\mu_1=\mu_3=\mu_6=\mu_{10}\) by \eqref{key condn}. Hence from \eqref{SC3_2.1}, we get \(\mu_{\text{min}}(\mathcal{E}|_{D_1^{(2)}})=\mu_{10}\). 

Similarly, one can see that \(\mu_{\text{min}}(\mathcal{E}|_{D_1'^{(2)}})=\mu_{10}.\)

Note that in this case we have \(x \notin \Gamma_3^{(2)} (=\Gamma_2 \subset X_2)\). Then using Theorem \ref{SC_H}, we have \begin{equation}\label{mu_10_1}
	\varepsilon(X_3^{(2)},\mathcal{E}|_{X_3^{(2)}},x)  \leq \mu_{10}. 
\end{equation}

Let \(C \in \mathcal{C}_{\mathcal{E}, x}\) be such that \(\pi_*C \not\subset X_3^{(2)}\). We show that 
\begin{equation*}
	\frac{\xi \cdot C}{m} \geq \mu_{10}.
\end{equation*}
Since \(x \in \Gamma_3^{(3)} \setminus \Gamma_3^{(2)}\), we have \(\pi_*C \not\subset D_3'\) by \eqref{curveGam}. Now, by \eqref{linequiv}, $$D_3' \,\sim_{\text{lin}}\, D_3-c_{1,3}D_1-c_{2,3} D_2.$$  Hence  
$$0 \,\le\, C\cdot D_3' \,=\, C\cdot(D_3-c_{1,3}D_1-c_{2,3} D_2) \,=\, p_3-c_{1,3}p_1-c_{2,3} p_2\, .$$ So $$p_3
\,\ge\, c_{1,3} \, p_1.$$
Since we have $\pi_* C \,\not\subset \, X_3^{(2)}$, again arguing as in the proof of Proposition \ref{SC X_3 p1}, we get \(p_1 \geq m\). 
From \eqref{X1}, we get 
\begin{equation*}
	\begin{split}
		&	\frac{\xi \cdot C}{m}  = \frac{a_0 + \sum_{j=1}^{12} a_j \, \mu_j}{m}\\
		& \geq \frac{a_0 + a_2 \, \mu_2 + a_4 \, \mu_4+ a_5 \, \mu_5 + a_8 \, \mu_8 + a_9 \, \mu_9 +p_3 \, \mu_{10}+ a_{11} \, \mu_{11} }{m} \ (\text{using \eqref{key condn2} and \eqref{X-coefficients}})\\
		& \geq \mu_{10} \ (\text{using the fact that } p_3 \geq m)\\
		& \geq \varepsilon(X_3^{(2)},\mathcal{E}|_{X_3^{(2)}},x)  \ (\text{by } \eqref{mu_10_1}).
	\end{split}
\end{equation*} 
 Hence the proposition follows from Lemma \ref{SC X3 lem}.
\end{proof}

\begin{proof}[Proof of Theorem \ref{SC_3}] The theorem follows immediately from Proposition \ref{SC X3 p2}  and Proposition \ref{SC X3 p3}.
\end{proof}

\begin{rmk}{\rm
		Let $\mc{E}$ be an equivariant nef vector bundle of rank \(r\) on \(X_3\) such that for any invariant curve \(l\), the following holds:
		$$\mu_{\text{min}}(\mathcal{E}|_{l})= a_1 \, \mu_5 + a_2 \, \mu_8 + a_3 \, \mu_{10} \text{ whenever } l \, \equiv a_1 \, l_5 + a_2 \, l_8 + a_3 \, l_{10}, \, a_1, a_2, a_3 \in \Z_{\geq 0}.$$ 
		Then $\mc{E}$ satisfies all the hypotheses of Theorem \ref{SC_3}.}
\end{rmk}

\noindent
We derive the following corollary as an immediate application.
\begin{cor}\label{hmplike}
	Let $\mathcal{E}$ be an equivariant nef vector bundle of rank \(r\) on \(X_3\) satisfying the conditions of Theorem \ref{SC_3}.
	Then the Seshadri constants of $\mathcal{E}$ at any \(x \in X_3\) are given by 	
	$$\varepsilon(X_3,\mathcal{E},x) =\,{\rm min}\left\{\mu_{\text{min}}(\mathcal{E}|_{\Gamma_3^{(i)}}) \,\mid\, x \in \Gamma_3^{(i)}\right\}.$$
\end{cor}
\begin{rmk}{\rm
Using similar methods, it is possible to compute Seshadri constants for $X_n$, $n \geq 4$, but the calculations are cumbersome.
}  
\end{rmk}
\begin{rmk}\label{hmp} {\rm 
		Seshadri constants for line bundles on toric varieties at torus fixed points are first computed by Di Rocco \cite{DiRocco}. In \cite{HMP}, the authors have generalized this result for equivariant vector bundles. Seshadri constants of an equivariant nef vector bundle $\mathcal{E}$ at a torus fixed point $x$ are computed via the restriction of $\mc{E}$ to the invariant curves passing through $x$ (see \cite[Proposition 3.2]{HMP}). Thus, Corollary \ref{hmplike} can be viewed as a generalization of these results for all points on Bott towers up to stage 3. 
	}
\end{rmk}

\begin{ex}{\rm  
		Consider the vector space \(E=\C^2\) and three distinct one dimensional subspaces \(L_1, L_2\) and \(L_{4}\) in \(E \). Now define the filtrations \( \left( E, \{ E^{v_j}(i) \}_{j=1, \, \ldots, \, 6} \right) \)  as follows:
		
		$
		E^{v_j}(i) = \left\{ \begin{array}{ccc}
			
			E& i \leqslant 0 \\
			
			L_j & i =1 \\ 
			
			0 & i > 1
		\end{array} \right. 
		$
		for \(j=1, 2, 4\) 
		and 
		$
		E^{v_j}(i) = \left\{ \begin{array}{ccc}
			
			E &  i \leq 0 \\
			
			0 &  i > 0
		\end{array} \right.
		$
		for \(j=3, 5, 6\).
		
		Hence, the filtrations \( \left( E, \{ E^{v_j}(i) \}_{j=1, \ldots, 6} \right) \) correspond to a rank \(2\) equivariant indecomposable vector bundle on \(X_3\), say \(\mathcal{E}\) (see \cite[Proposition 6.1.1]{DDK}). 
		
		Let $\sigma_{i,j, k}$ denote the maximal cone generated by the primitive vectors \(v_i, v_j, v_k\) in the fan $\Delta_3$ of \(X_3\). On each maximal cone, by Klyachko's compatibility condition \eqref{KCC}, the associated characters are given by 
		\begin{equation*}
			\begin{split}
				&		\mathbf{u}(\sigma_{1,2,3})=\{(1,0,0), (0,1,0)\}, \mathbf{u}(\sigma_{1,2,6})=\{(1,0,0), (0,1,0)\}, \\
				&		\mathbf{u}(\sigma_{3, 4, 5})=\{(-1,0,0), (0,0,0)\},\mathbf{u}(\sigma_{1,5,6})=\{(1,0,0), (0,0,0)\},\\
				& \mathbf{u}(\sigma_{4,5,6})=\{(-1,0,0), (0,0,0)\},  \mathbf{u}(\sigma_{2,4,6})=\{(c_{1,2}, 1,0), (-1,0,0)\},\\
				&\mathbf{u}(\sigma_{1, 3,5})=\{(1,0,0), (0,0,0)\}, \mathbf{u}(\sigma_{2,3,4})=\{(c_{1,2}, 1,0), (-1,0,0)\}.  
			\end{split} 
		\end{equation*}
Then following Subsection \ref{subsec5}, restrictions of $\mathcal{E}$ to the invariant curves are given by
		\begin{equation*}
			\begin{split}
				&	\mathcal{E}|_{l_1}=\mathcal{O}_{\mathbb{P}^1} \oplus \mathcal{O}_{\mathbb{P}^1} , 	\mathcal{E}|_{l_2}=\mathcal{O}_{\mathbb{P}^1}\oplus \mathcal{O}_{\mathbb{P}^1} (1), \mathcal{E}|_{l_3}=\mathcal{O}_{\mathbb{P}^1} \oplus \mathcal{O}_{\mathbb{P}^1} , \\
				& 		\mathcal{E}|_{l_4}=\mathcal{O}_{\mathbb{P}^1} \oplus \mathcal{O}_{\mathbb{P}^1} (1), \mathcal{E}|_{l_5}=\mathcal{O}_{\mathbb{P}^1}(-c_{1,2}) \oplus \mathcal{O}_{\mathbb{P}^1} (2), \mathcal{E}|_{l_6}=\mathcal{O}_{\mathbb{P}^1} \oplus \mathcal{O}_{\mathbb{P}^1} , \\
					&	\mathcal{E}|_{l_7}=\mathcal{O}_{\mathbb{P}^1}(-c_{1,2}) \oplus \mathcal{O}_{\mathbb{P}^1} (2), 	\mathcal{E}|_{l_8}=\mathcal{O}_{\mathbb{P}^1} \oplus \mathcal{O}_{\mathbb{P}^1} (1), \mathcal{E}|_{l_9}=\mathcal{O}_{\mathbb{P}^1} \oplus \mathcal{O}_{\mathbb{P}^1} (2), \\
				& 		\mathcal{E}|_{l_{10}}=\mathcal{O}_{\mathbb{P}^1} \oplus \mathcal{O}_{\mathbb{P}^1} , \mathcal{E}|_{l_{11}}=\mathcal{O}_{\mathbb{P}^1} \oplus \mathcal{O}_{\mathbb{P}^1} (1), \mathcal{E}|_{l_{12}}=\mathcal{O}_{\mathbb{P}^1} \oplus \mathcal{O}_{\mathbb{P}^1} (2).
			\end{split}
		\end{equation*}
		So $\mathcal{E}$ is not nef (by Proposition \ref{HMPT2.1}). Consider \(D=a_1 D_1 + a_2 D_2 + a_3 D_3\) with \(a_1 \geq c_{1,2}, \, a_2 \geq 0\) and \(a_3 \geq a_2\). Then 
		\begin{equation*}
			\begin{split}
				& D \cdot l_1=a_3, \, D \cdot l_2=a_2, \, D \cdot l_3=a_3, \, D \cdot l_4=a_2+ c_{2,3} \, a_3, \, \\
				& D \cdot l_5=a_1, \, D \cdot l_6=a_3, \, D \cdot l_7= a_1+ c_{1,3} \, a_3, \, D \cdot l_8=a_2, \, \\
				& D \cdot l_9=a_1 + c_{1,2} \, a_2, \,D \cdot l_{10}=a_3, \, D \cdot l_{11}=a_2 + c_{2,3} \, a_3, \, \\
				& D \cdot l_{12}=a_1 + c_{1,2} \, a_2 + (c_{1,3} + c_{1,2} \, c_{2,3} ) \, a_3.
			\end{split}
		\end{equation*}
		Thus restrictions of $\mathcal{E}\otimes \mathcal{O}(D)$ to the invariant curves are given by
		\begin{equation*}\label{H10}
			\begin{split}
				&	(\mathcal{E} \otimes \mathcal{O}(D))|_{l_1}=\mathcal{O}_{\mathbb{P}^1}(a_3) \oplus \mathcal{O}_{\mathbb{P}^1}(a_3) , 	
				(\mathcal{E} \otimes \mathcal{O}(D))|_{l_2}=\mathcal{O}_{\mathbb{P}^1}(a_2) \oplus \mathcal{O}_{\mathbb{P}^1} (1+a_2), \\
				&	(\mathcal{E} \otimes \mathcal{O}(D))|_{l_3}=\mathcal{O}_{\mathbb{P}^1}(a_3) \oplus \mathcal{O}_{\mathbb{P}^1}(a_3) , \\
		 	&	(\mathcal{E} \otimes \mathcal{O}(D))|_{l_4}=\mathcal{O}_{\mathbb{P}^1}(a_2+ c_{2,3} \, a_3) \oplus \mathcal{O}_{\mathbb{P}^1} (1+a_2+ c_{2,3} \, a_3), \\
		&	(\mathcal{E} \otimes \mathcal{O}(D))|_{l_5}=\mathcal{O}_{\mathbb{P}^1}(a_1-c_{1,2}) \oplus \mathcal{O}_{\mathbb{P}^1} (2+a_1), \\
		&	(\mathcal{E} \otimes \mathcal{O}(D))|_{l_6}=\mathcal{O}_{\mathbb{P}^1}(a_3) \oplus \mathcal{O}_{\mathbb{P}^1}(a_3) , \\
		&	(\mathcal{E} \otimes \mathcal{O}(D))|_{l_7}=\mathcal{O}_{\mathbb{P}^1}(a_1+ c_{1,3} \, a_3-c_{1,2}) \oplus \mathcal{O}_{\mathbb{P}^1} (a_1+ c_{1,3} \, a_3+2), 	\\
		&	 (\mathcal{E} \otimes \mathcal{O}(D))|_{l_8}=\mathcal{O}_{\mathbb{P}^1} (a_2)\oplus \mathcal{O}_{\mathbb{P}^1} (1+a_2), \\
		&	(\mathcal{E} \otimes \mathcal{O}(D))|_{l_9}=\mathcal{O}_{\mathbb{P}^1}(a_1 + c_{1,2} \, a_2) \oplus \mathcal{O}_{\mathbb{P}^1} (2+a_1 + c_{1,2} \, a_2), \\
		&			(\mathcal{E} \otimes \mathcal{O}(D))|_{l_{10}}=\mathcal{O}_{\mathbb{P}^1}(a_3) \oplus \mathcal{O}_{\mathbb{P}^1}(a_3) , \\
		&	(\mathcal{E} \otimes \mathcal{O}(D))|_{l_{11}}=\mathcal{O}_{\mathbb{P}^1}(a_2 + c_{2,3} \, a_3) \oplus \mathcal{O}_{\mathbb{P}^1} (1+a_2 + c_{2,3} \, a_3), \\
		&	(\mathcal{E} \otimes \mathcal{O}(D))|_{l_{12}}=\mathcal{O}_{\mathbb{P}^1}(a_1 + c_{1,2} \, a_2 + (c_{1,3} + c_{1,2} \, c_{2,3} ) \, a_3) \\
		& \hspace{2.8 cm} \oplus \mathcal{O}_{\mathbb{P}^1} (2+a_1 + c_{1,2} \, a_2 + (c_{1,3} + c_{1,2} \, c_{2,3} ) \, a_3).
			\end{split}
		\end{equation*}
		This implies that \(\mathcal{E}\otimes \mathcal{O}(D)\) is nef (by Proposition \ref{HMPT2.1}). Also, note that
		\begin{equation*}
			\begin{split}
				& \mu_1=\mu_3=\mu_6=\mu_{10}=a_3; \, \mu_2=\mu_8=a_2; \, \mu_4=a_2 + c_{2,3} \, a_3, \, \mu_5=a_1-c_{1,2},\\
				& \mu_7=a_1 + c_{1,3} \, a_3-c_{1,2}, \, \mu_9=a_1+c_{1,2} \, a_2, \, \mu_{11}=a_2+ c_{2,3} \, a_3,\\
				& \mu_{12}=a_1 + c_{1,2} \, a_2 + (c_{1,3} + c_{1,2} \, c_{2,3} ) \, a_3.
			\end{split}
		\end{equation*}
		So \(\mathcal{E}\otimes \mathcal{O}(D)\) satisfies the conditions of Theorem \ref{SC_3}. Hence for \(x \in X_3\), the Seshadri constant is given by 
	
	\begin{equation*}
		\varepsilon(\mathcal{E}\otimes \mathcal{O}(D),x) =
		\begin{cases}
				\text{min}\{a_1-c_{1,2},\, a_2 \}, & {\rm if}\ x\in \Gamma_3, \\
				a_2, & {\rm if}\ x\notin \Gamma_3, \, x \in \Gamma_3^{(2)},\\
             a_3,  & {\rm if}\ x\notin \Gamma_3, \, x \notin \Gamma_3^{(2)}.
				\end{cases}
	\end{equation*} 

		Again, for a suitable choice of polarization on \(X_3\), the vector bundle \(\mathcal{E}\) becomes stable (see \cite[Proposition 6.1.5]{DDK}).  We observe that discriminant of $\mathcal{E}$ is non-zero using the same argument as in Example \ref{hirz-ex}.
		
}
\end{ex}


\end{document}